\documentclass[letterpaper,12pt]{article}
\makeatletter
\usepackage[dvipsnames]{xcolor}
\usepackage{graphicx} 
\usepackage{enumitem}
\usepackage[all]{xy}
\usepackage{amsmath, amssymb, amsthm}
\usepackage{thmtools}
\usepackage[left=1in, right=1in, top=1in, bottom=1in]{geometry}

\author{Sara Stephens}
\usepackage{float}
\def\makelowercaseFrak#1{%
\expandafter\newcommand\csname mf#1\endcsname{\mathfrak{#1}}}
\count@=0
\loop
\advance\count@ 1
\edef\y{\@alph\count@} 
\expandafter\makelowercaseFrak\y
\ifnum\count@<26
\repeat

\usepackage{xfrac}
\usepackage[modulo]{lineno}
\usepackage[colorlinks=true,linkcolor=blue,citecolor=blue,urlcolor=blue,citebordercolor={0 0 1},urlbordercolor={0 0 1},linkbordercolor={0 0 1}]{hyperref} 

\usepackage[alphabetic]{amsrefs}
\usepackage{tikz-cd}
\usepackage{tikz}
\usetikzlibrary{arrows, arrows.meta, positioning}
\usetikzlibrary{cd}
\usepackage[nameinlink]{cleveref}

\title{Strictly Semistable Quasimaps on the Wall}
\date{June 2026}

\theoremstyle{plain}
\newtheorem{thm}{Theorem}[section]

\newtheorem{lem}[thm]{Lemma}

\newtheorem{prop}[thm]{Proposition}

\newtheorem{const}[thm]{Construction}

\theoremstyle{definition}
\newtheorem{rem}[thm]{Remark}
\newtheorem{defn}[thm]{Definition}

\newtheorem{notn}[thm]{Notation}
\newtheorem{ex}[thm]{Example}

\newtheorem*{thm*}{Theorem}
\newtheorem*{prop*}{Proposition}

\DeclareMathOperator{\Spec}{Spec}

\renewcommand{\ss}{\mathrm{ss}}

\newcommand{\bG}{\mathbb{G}}
\newcommand{\bZ}{\mathbb{Z}}
\newcommand{\bN}{\mathbb{N}}
\newcommand{\bQ}{\mathbb{Q}}
\newcommand{\bR}{\mathbb{R}}
\newcommand{\bC}{\mathbb{C}}
\newcommand{\bA}{\mathbb{A}}
\newcommand{\bP}{\mathbb{P}}
\newcommand{\cL}{\mathcal{L}}

\newcommand{\cT}{\mathcal{T}}
\newcommand{\cY}{\mathcal{Y}}
\newcommand{\cC}{\mathcal{C}}
\newcommand{\cO}{\mathcal{O}}
\newcommand{\fY}{\mathfrak{Y}}
\newcommand{\fX}{\mathfrak{X}}

\newcommand{\git}{\mathbin{
  \mathchoice{/\mkern-6mu/}
    {/\mkern-6mu/}
    {/\mkern-5mu/}
    {/\mkern-5mu/}}}

\DeclareMathOperator{\ev}{ev}

\DeclareMathOperator{\Perf}{Perf}

\DeclareMathOperator{\coker}{coker}

\DeclareMathOperator{\Pic}{Pic}

\DeclareMathOperator{\univ}{univ}

\DeclareMathOperator{\Map}{Map}

\DeclareMathOperator{\Aut}{Aut}

\DeclareMathOperator{\Spf}{Spf}

\DeclareMathOperator{\ch}{ch}
\DeclareMathOperator{\tot}{tot}

\DeclareMathOperator{\Proj}{Proj}

\DeclareMathOperator{\ADeg}{ADeg}
\DeclareMathOperator{\Filt}{Filt}
\DeclareMathOperator{\Bl}{Bl}
\DeclareMathOperator{\Qcoh}{QCoh}
\DeclareMathOperator{\wt}{wt}

\DeclareMathOperator{\QmapO}{\frak{QMap}}

\newcommand{\Qmap}{\QmapO^\epsilon_{g, n}(\bP^N, d)}
\newcommand{\QmapPre}{\QmapO_{g, n}^{\rm{pre}}(\bP^N, d)}
\newcommand{\Qmapplus}{\QmapO^{\epsilon_+}_{g, n}(\bP^N, d)}
\newcommand{\Qmapminus}{\QmapO^{\epsilon_-}_{g, n}(\bP^N, d)}
\newcommand{\bigQmap}{\QmapO^{\epsilon_0-\ss}_{g, n}(\bP^N, d)}

\DeclareMathOperator{\STr}{\overline{ST}_R}

\theoremstyle{plain}
\newtheorem*{theoremA}{Theorem A (= \Cref{P:alg-stack} + \Cref{T:finite-type} + \Cref{T:aff-diag} + \Cref{thm A})}
\newtheorem*{theoremB}{Theorem B (= \Cref{T:theta-stratifications} + \Cref{prop:ss=Q-minus} + \Cref{prop:ss=Q-plus})}

\begin{document}

\maketitle
\begin{abstract}

Moduli spaces of $\epsilon$-stable quasimaps exhibit a wall-and-chamber structure, interpolating between the moduli of stable quasimaps and Kontsevich’s moduli spaces of stable maps. In the case where the target is projective space, we develop an intrinsic framework for the wall-crossing phenomenon by constructing an algebraic stack where strictly semistable objects appear at a wall in the space of $\epsilon$-stability conditions. We show this algebraic stack admits a proper good moduli space and analyze the variation of $\Theta$-stratification on this stack. This yields a new approach to developing a K-theoretic wall-crossing formula for $\epsilon$-stable quasimap invariants via non-abelian  localization.
\end{abstract}
{\small\tableofcontents}
\newpage
\section{Introduction}

A classical problem in enumerative geometry asks for the number of rational curves contained in an algebraic variety that satisfy prescribed incidence conditions. In modern enumerative geometry, rigorous answers are provided by curve counting theories, which give a systematic framework that assigns virtual numbers to curves (of fixed genus and degree) inside a given smooth projective variety $X$. Prominent among these is Gromov--Witten theory, whose central object is the moduli space $\overline{\mathcal{M}}_{g,n}(X, d)$ of degree $d$ stable maps from an $n$-marked nodal curve of genus $g$ into $X$. This moduli space is a proper Deligne–Mumford stack with a virtual fundamental class, and the Gromov--Witten invariants are defined as integrals over it. The open substack $\mathcal{M}_{g,n}(X, d) \subset \overline{\mathcal{M}}_{g,n}(X, d)$ of maps with smooth domain curve is not compact; other natural compactifications exist, yielding different invariants closely related to Gromov--Witten invariants of \cite{KM}.

A family of such compactifications is provided by the theory of $\epsilon$-stable quasimaps, established by Ciocan-Fontanine, Kim, and Maulik in \cite{ciocanfontanine2011stablequasimapsgitquotients}, for a large class of GIT quotients of affine varieties. This theory unifies many earlier constructions \cites{kontsevich, MOP, Ciocan_Fontanine_2010, Toda_2011, mustata2007intermediate}. The moduli stacks $\frak{QMap}^\epsilon_{g, n}(X, d)$ are proper and separated Deligne--Mumford stacks of finite type, with stability depending on a positive rational parameter $\epsilon$.  As $\epsilon$ varies, these spaces $\frak{QMap}^\epsilon_{g, n}(X, d)$ interpolate between two classical moduli problems: stable maps and stable quasimaps. As $\epsilon \rightarrow \infty$, the theory recovers the usual Gromov–Witten theory of $X$, and the quasimaps coincide with the familiar stable maps, where the source curve has many rational tails and no basepoints. On the other hand, when $\epsilon \rightarrow 0$, we obtain stable quasimaps, coinciding with stable quotients defined by Marian, Oprea, and Pandharipande \cite{MOP}. These quasimaps have no rational tails and many basepoints. The simpler domain curves allow certain quasimap invariants to be explicitly computed, especially in genus 0. 

As $\epsilon$ varies from $0$ to $\infty$,  the space $\mathbb{Q}_{> 0} \cup \{\infty\}$ of stability parameters admits a wall-and-chamber structure. The theory changes only at discrete values of $\epsilon$, called walls. Assuming for simplicity that $(g, n)$ is not $ (0, 0)$, the walls occur at $\epsilon_0 = 1/m_0$ with $m_0 \in \bN$. For fixed numerical data $(g, n, d)$, there are a finite number of walls, and the moduli spaces are constant in each chamber $(\frac{1}{m_0+1}, \frac{1}{m_0}]$. As $\epsilon$ increases and we move through the intermediate moduli spaces, the domain curve is allowed to develop more rational components and the quasimap is closer to being a genuine map. 

Wall-crossing formulas express changes in the enumerative invariants that occur as $\epsilon$ varies between adjacent chambers. Developing such formulas has been a subject of much interest: a cohomological wall-crossing formula was originally conjectured by Ciocan-Fontanine and Kim \cite[Conjecture 1.1]{CKwall-crossings-and-mirror-symmetry}, and proven in increasing generality in \cites{CiocanFontanineKim2014, CKhigher-genus, clader-janda-ruan, CKwall-crossings-and-mirror-symmetry}, and full generality in \cite{Zhou:masterspace}. An analogous K-theoretic wall crossing formula appears in \cite{zhang2020ktheoreticquasimapwallcrossing} in full generality.

\paragraph{An intrinsic perspective on wall-crossing.} In this paper, we take a new approach to studying the wall-crossing behavior of $\epsilon$-stable quasimaps using algebraic stacks. We follow the Beyond GIT program \cite{halpernleistner2022structureinstabilitymodulitheory}, which presents an intrinsic approach to construct and study good moduli spaces and $\Theta$-stratifications. One application of this theory is a generalization of the classical theory of variation of GIT quotient \cite{VGIT}, by varying the $\Theta$-stratification of an algebraic stack $\frak{X}$. A complete discussion can be found in \cite[Section 5.6]{halpernleistner2022structureinstabilitymodulitheory}. We summarize the wall-crossing principle here:

Let $\frak{X}_1$ and $\frak{X}_2$ be algebraic stacks parameterizing objects of geometric interest, both admitting good moduli spaces $M_1$ and $M_2$ along with a birational morphism of algebraic spaces $M_1 \dashrightarrow M_2$. Ideally, there should be a third algebraic stack $\frak{X}$ that admits a good moduli space $M$, also parameterizing objects of geometric interest. The stack $\frak{X}$ should contain both $\frak{X}_1$ and $\frak{X}_2$ as open substacks, with the resulting maps $M_1 \rightarrow M \leftarrow M_2$ being projective and birational. Defining $\frak{X}$ requires some care, as $\frak{X}$ must necessarily have objects with positive-dimensional stabilizers. Constructing moduli spaces for true algebraic stacks is much more subtle than Deligne--Mumford stacks.

By varying the $\Theta$-stratification of $\frak{X}$, one can recover the open substacks $\frak{X}_i \subset \frak{X}$ for $i = 1, 2$ as the semistable locus of the stratification, thereby enabling a comparison of the geometries of $\frak{X}_1$ and $\frak{X}_2$.
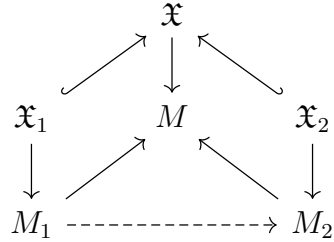
\begin{center}
\begin{tikzcd}
                                  & \frak{X} \arrow[d] &                                 \\
\frak{X}_1 \arrow[ru, hook] \arrow[d]   & M                  & \frak{X}_2 \arrow[lu,hook'] \arrow[d] \\
M_1 \arrow[ru] \arrow[rr, dashed] &                    & M_2 \arrow[lu]                 
\end{tikzcd}
\end{center}

This variation of good moduli space picture has been used in \cites{bu2025intrinsicdonaldsonthomastheoryi, bu2025intrinsicdonaldsonthomastheoryii}, where the authors use this setup in order to develop a combinatorial framework for wall-crossing in Donaldson--Thomas theory. This geometric picture is absent from quasimap story -- there is a missing theory of semistable quasi-maps for $\epsilon_0 = 1/m_0$. The definition of $\epsilon$-stability in \cite{ciocanfontanine2011stablequasimapsgitquotients} is constructed in such a way that stability for quasimaps on the wall $\epsilon_0 = 1/m_0$ is the same as stability in the chamber $(\frac{1}{m_0+1}, \frac{1}{m_0}]$ just to the left of the wall. There is a missing notion of semistability for $\epsilon_0$-semistable quasimaps on the wall, which is needed to apply the above framework for wall-crossing in moduli theory.

\paragraph{Main results.} We apply the wall-crossing principle discussed above to study wall-crossing for $\epsilon$-stable quasimaps, in the case where the target is projective space. Let $\epsilon_0 = 1/m_0$ be a wall in the space of stability conditions, $\epsilon_-$ denote a rational number in the chamber to the left of $\epsilon_0$, and similarily let $\epsilon_+$  denote a rational number in the chamber to the right of $\epsilon_0$. We construct a stability condition termed $\epsilon_0$-semistability in \Cref{D:e0-stability}, which defines an algebraic stack $\bigQmap$.

\begin{theoremA} $\bigQmap$ is a finite-type algebraic stack with affine diagonal, containing $\Qmapplus$ and $\Qmapminus$ as open substacks. It admits a proper good moduli space.
\end{theoremA}

The stacks $\bigQmap$, $\Qmapplus$ and $\Qmapminus$ play the roles of $\frak{X}$, $\frak{X}_1$ and $\frak{X}_2$ above. We use this setup to study wall-crossing.

\begin{theoremB} There exists a pair of numerical invariants $\nu_1, \nu_2$ arising from explicit cohomology classes $\ell_1, \ell_2 \in H^2(\bigQmap, \bQ)$ and a norm $b \in H^4(\bigQmap, \bQ)$ 
which define \\$\Theta$-stratifications of $\bigQmap$. The semistable loci of these stratifications recover the moduli stacks on either side of the wall:
$$\bigQmap^{\ell_1-\ss} = \Qmapminus, \hspace{0.7cm} \bigQmap^{\ell_2-\ss} = \Qmapplus,$$
where $\epsilon_- < \epsilon_0 < \epsilon_+$ are stability parameters in adjacent chambers.
\end{theoremB}

\paragraph{Plan of the paper.} In \Cref{S:preliminaries}, we review the construction of $\epsilon$-stable quasimaps. In \Cref{S:moduli-stack}, we define $\bigQmap$ and discuss it's properties. The effect of crossing the wall $\epsilon_0$ from the right to the left is to replace degree $m_0$ rational tails by length $m_0$ base points, so $\bigQmap$ includes quasimaps with both degree $m_0$ rational tails and length $m_0$ base points. In particular, the moduli stack contains strictly semistable objects with $(\bC^*)^s$ automorphism groups, namely quasimaps with $s$ degree $m_0$ rational tails that each contain a length $m_0$ basepoint.


The proof of Theorem A uses necessary and sufficient conditions for an algebraic stack of finite type with affine diagonal to admit a good moduli space, developed in \cite{Alper_2023_Existence}. These are codimension-two filling conditions, termed $\Theta$-reductivity and $S$-completeness. We analyze these conditions in \Cref{S:GMS}, which is the main technical part of the paper, and additionally establish the existence part of the valuative criterion for properness.

As an application of studying $\Theta$-reductivity and $S$-completeness, we are able to connect our construction of $\epsilon_0$-semistable quasimap fillings to the affine degeneration space, introduced by Halpern-Leistner in \cite[Section 4]{ICM2026}. This is discussed in \Cref{S:ADeg}.

Variation of $\Theta$-stratification on $\bigQmap$ is discussed in \Cref{S:theta-stratifications}. We introduce a family of natural line bundles on $\bigQmap$ to obtain two different numerical invariants providing a numerical measure of instability of points of $\bigQmap$. An analysis of semistability is carried out, leading to Theorem B.

\paragraph{Future directions.}

\begin{enumerate}[leftmargin=*]
    \item \textbf{An explicit wall crossing formula.} The most prominent application of the present work is a wall-crossing formula relating $\epsilon$-stable quasimap invariants from different stability chambers. We have exactly the moduli-theoretic setup needed to apply the wall-crossing formula for algebraic stacks admitting a proper good moduli space as presented in \cite[Section 4.1]{halpernleistner2025categoricalperspectivenonabelianlocalization} using the tools of  non-abelian localization. Applying the formula \cite[Equation 18]{halpernleistner2025categoricalperspectivenonabelianlocalization} to the two $\Theta$-stratifications of $\bigQmap$ of Theorem B gives for any $F \in \Perf(\bigQmap)$
    \begin{align*}
        \chi(&\Qmapminus, F) - \chi(\Qmapplus, F) = \label{eq:wall-crossing}\\
        &\sum_{\beta} \chi\left(\mathcal{Z}_{\beta}^{\ell_2}, \frac{\tot^*_\beta(F)}{e(N_{\mathcal{Z}_{\beta}^{\ell_2}}\bigQmap)}\right) - \sum_{\alpha} \chi\left(\mathcal{Z}_{\alpha}^{\ell_1}, \frac{\tot^*_\alpha(F)}{e(N_{\mathcal{Z}_{\alpha}^{\ell_1}}\bigQmap)}\right).
    \end{align*}
    Here $\tot_\alpha : \mathcal{Z}_{\alpha}^{\ell_1} \xrightarrow{} \bigQmap$ and $\tot_\beta : \mathcal{Z}_{\beta}^{\ell_2} \xrightarrow{} \bigQmap$ are the centres of the unstable strata for $\ell_1$ and $\ell_2$, respectively. A more explicit description of this formula in terms of $\epsilon$-stable quasimap invariants will appear in forthcoming work.
    \item \textbf{More general targets $X/G$, including orbifold targets.} The theory of $\epsilon$-stable quasimaps in \cite{CiocanFontanineKim2014} is developed for all GIT quotients $X/G$, where $X$ is an affine variety with at worst lci singularities and $G$ is a reductive group. We expect our results to extend to these general targets. We are currently working to generalize our approach to all such $X/G$, and to extend the wall crossing formula above. We also expect our results to hold for the theory of $\epsilon$-stable quasimaps to orbifold targets, developed in \cite{cheong2015orbifoldquasimaptheory}.
\end{enumerate}

\paragraph{Related work.} The most notable comparison to our work is that of Zhang and Zhou \cite{zhang2020ktheoreticquasimapwallcrossing}, which builds on earlier work of Zhou \cite{Zhou:masterspace}. They prove a K-theoretic wall-crossing formula relating $\epsilon$-stable quasimap invariants across different stability chambers for all targets and all genera, including the orbifold case. Their proof uses virtual localization on a larger moduli space, called a “master space”. The fixed-point loci of the master space are in one-to-one correspondence with the terms in the wall-crossing formula. 

The master space constructed in \cite{Zhou:masterspace} plays a similar role to $\bigQmap$, but it is built very differently. 
It is a $\bP^1$-bundle over the stack of ``entangled" tails, where the entanglement comes from a sequence of blowups on the moduli stack of weighted curves. This results in the master space being a proper Deligne--Mumford stack with a global $\bC^*$-action, as opposed to our algebraic stack $\bigQmap$ which allows objects with higher rank torus automorphism groups. Our construction is more directly motivated by the inherent geometry of the wall crossing, and avoids the auxiliary construction of a master space.

\paragraph{Acknowledgements.} I would like to thank my advisor, Daniel Halpern-Leistner, for his guidance and helpful discussions during the preparation of this paper. I am also grateful to Andres Fernandez Herrero, Giovanni Inchiostro and Rachel Webb for many useful conversations. Lastly, I'd like to thank Alekos Robotis for his helpful feedback on drafts of this paper. The author is partially supported by a Natural Sciences and Engineering Research Council of Canada (NSERC) PGS-D scholarship.

\paragraph{Conventions and Notation.} We work over $\bC$ throughout the paper. All schemes are locally of finite type. We sometimes write DM in place of Deligne--Mumford (e.g. DM stack). We will always use $R$ to denote a discrete valuation ring with field of fractions $K$ and uniformizer $\pi$.

\section{Preliminaries on $\mathbf{\epsilon}$-Stable Quasimaps}\label{S:preliminaries}

We recall the notion of a quasimap to $\bP^N$, given in \cite[Definition 3.1.1]{Ciocan_Fontanine_2010}. Let $\bC^*$ act on $\bC^{N+1}$ by $\lambda \cdot (y_0, \ldots, y_N) = (\lambda y_0, \ldots, \lambda y_N)$. Taking $\theta:\bC^*\xrightarrow{} \bC^*$ to be the identity character, the semistable locus is $(\bC^{N+1})^\ss = \bC^{N+1}\backslash \{0\}$ and the GIT quotient $\bC^{N+1} \git_\theta \bC^*$ is $ \bP^N$.

\begin{defn}\label{D:prestableQmap}
A $n$-pointed \textbf{prestable quasimap} to $\bP^N$ of genus $g$ and degree $d$ consists of data $(C, x_1, . . . , x_n, L, \{s_i\}_{i = 0}^N)$ where

\begin{itemize}
    \item $C$ is a connected nodal curve of genus $g$ with marked points $x_1, . . . , x_n$
    \item $L$ is a degree $d$ line bundle over $C$ with $s_0, \ldots s_N \in \Gamma(C, L)$
    \item the set of zeroes of $(s_0, \ldots s_N)$ is finite and disjoint from the nodes and the markings on $C$.
\end{itemize}
\end{defn}

\noindent This definition extends to families as follows.

\begin{defn}
A $n$-pointed \textbf{family of prestable quasimaps} to $\bP^N$ of genus $g$ and degree $d$ over scheme $S$ is given by $(\pi:\cC_S \rightarrow S, \{x_i: S \xrightarrow{} \cC_S\}_{i = 1}^n,\cL_S, \{s_j\}_{i = 0}^N)$ where

\begin{itemize}
    \item $\pi:\cC_S \rightarrow S$ is a flat family of connected nodal curves over $S$
    \item $x_1, \ldots x_n$ are disjoint sections of $\pi$
    \item $\cL_S$ is a line bundle over $\cC_S$ of degree $d$  along fibres of $\cC_S/S$, with each $s_j \in \Gamma(\cC_S, \cL_s)$
\end{itemize}    

such that restriction to each geometric fiber of $\pi$ is a prestable quasimap of genus $g$ and degree $d$ as in \Cref{D:prestableQmap}.
\end{defn}

We denote the stack of all $n$-pointed prestable quasimaps of degree $d$ and genus $g$ to $\bP^N$ by $\QmapPre$. We also sometimes use $\mu$ denote the map to $\bP^N$ given by the sections $ \{s_j\}_{i = 0}^N$.

\begin{defn} \label{D:e-stability} \cite[Definition 7.1.3]{ciocanfontanine2011stablequasimapsgitquotients}
    Let $\epsilon$ be a positive rational number. The prestable quasimap $(C, x_1, . . . , x_n, L, \{s_i\}_{i = 0}^N)$ to $[\bC^{N+1}/\bC^*]$ is \textbf{$\mathbf{\epsilon}$-stable} if

    \begin{enumerate}
        \item $\omega_C(x_1 + \ldots + x_n) \otimes L^\epsilon$ is ample, and
        \item For every point $x \in C$ we have $\ell(x) \leq 1/\epsilon$, where $\ell(x)$ is the \textit{length} of the point $x$, given by
        $$\ell(x) := \mathrm{length}_x(\coker(\cO_C^{\oplus N+1} \xrightarrow{} L)).$$
    \end{enumerate}
\end{defn}

A family of prestable quasimaps is $\epsilon$-stable if every geometric fiber is $\epsilon$-stable. We denote the stack of $n$-pointed $\epsilon$-stable quasimaps of genus $g$ and degree $d$ to $\bP^N$ by $\Qmap$.

\begin{thm}\label{T:e-stableDM}\cite[Theorem 7.1.6]{ciocanfontanine2011stablequasimapsgitquotients}
$\Qmap$ is a proper
separated Deligne--Mumford stack of finite type. 
\end{thm}

\begin{notn}
    We always consider $(g, n) \notin \{(0, 0), (0,1)\}$. Fix $\epsilon_0 = 1/m_0$ as a wall in the space of stability parameters, with $m_0 \in \bN$. Let $\epsilon_-$ denote a rational number in the chamber $(\frac{1}{m_0+1}, \frac{1}{m_0}]$ to the left of the $\epsilon_0$ wall. Similarly, we take $\epsilon_+ \in (\frac{1}{m_0}, \frac{1}{m_0-1}]$.
\end{notn}

To describe how the geometry of the source curve changes when crossing the wall $\epsilon_0$, we will use the following terminology for rational components.

\begin{notn}
    We say that a component of a curve is a \textbf{rational tail} if it is isomorphic to $\bP^1$ and intersects the rest of the curve at exactly one node and has no marked points. A \textbf{rational bridge} is a component isomorphic to $\bP^1$ that either intersects the rest of the curve in exactly two nodes with no marked points, or intersects the rest of the curve at only one node and has exactly one marked point.
\end{notn}

Assuming that $(g, n) \notin \{(0, 0), (0,1)\}$, the moduli spaces are non-empty and stay constant in each chamber $(\frac{1}{m_0+1}, \frac{1}{m_0}]$ for fixed numerical data $(g, n, d)$. For $\epsilon \in (1, \infty)$, the stability condition forces the sections $\{s_i\}_{i=0}^N$ to have no common zeroes, thus defining a genuine map to $\bP^N$ and recovering the notion of a Kontsevich stable map. On the other hand, when $\epsilon \le 1/d$, we have stable $0+$ quasimaps with no rational tails and lots of basepoints. Analyzing the stability condition of \Cref{D:e-stability} on either side of an $\epsilon_0 = 1/m_0$ wall gives

\begin{center}
\begin{minipage}[t]{0.4\textwidth}
\textbf{$\boldsymbol{\epsilon_-}$-Stable}
\begin{itemize}[leftmargin=*]
    \item $\deg(L|_{\text{rational tails}}) \ge m_0+1$
    \item $\deg(L|_{\text{rational bridges}}) > 0$
    \item $\ell(x) \le m_0$ for every $x \in C$
\end{itemize}
\end{minipage}
\hspace{0.5cm} 
\begin{minipage}[t]{0.5\textwidth}
\textbf{$\boldsymbol{\epsilon_+}$-Stable}
\begin{itemize}[leftmargin=*]
    \item $\deg(L|_{\text{rational tails}}) \ge m_0$
    \item $\deg(L|_{\text{rational bridges}}) > 0$
    \item $\ell(x) \le m_0-1$ for every $x \in C$
\end{itemize}
\end{minipage}
\end{center}
\hspace{1cm}

Specifically, as we cross a $1/m_0$ wall from the $\epsilon_-$ chamber to the $\epsilon_+$ chamber, basepoints of length $m_0$ fail to be stable, and are replaced by degree $m_0$ rational tails.

\begin{prop}
    As an $\epsilon_0 = 1/m_0$ wall is crossed from the $\epsilon_-$ chamber to the $\epsilon_+$ chamber, the source curve can only develop a single rational tail at a basepoint of length $m_0$, and not a more complicated rational tree.
\end{prop}
    \begin{proof}
        Let $C'$ be a component of $C$. Suppose the degree of the quasimap on the component $C'$ is $d'$, and there is a length $m_0$ basepoint on $C'$, so $m_0 \le d'$. As you cross the wall, the basepoint becomes unstable, and the curve develops a tree at the former basepoint. The degree on the component $C'$ becomes $d'-m_0$, and the new tree has total degree $m_0$. Suppose a more complicated tree than just a single $\bP^1$ bubble occurs. Then, any rational tail must have degree $\ge m_0$ under the $\epsilon_+$ stability condition, forcing rational bridges in the tree to have non-positive degree, which contradicts $\epsilon_+$-stability.
    \end{proof}

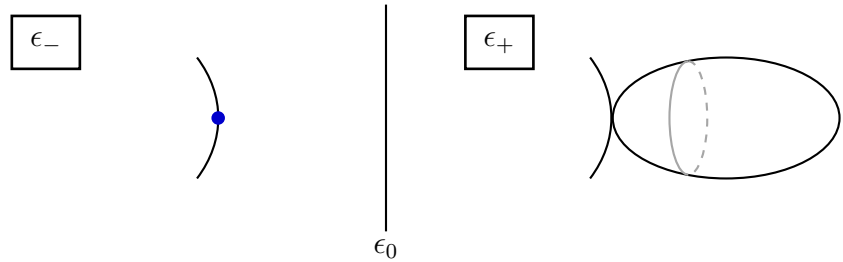
\begin{figure}[H]
    \centering
\begin{tikzpicture}[scale=1.0]
  \tikzset{
    surface/.style={thick},
    cross_front/.style={gray!70, thick},
    cross_back/.style={gray!70, thick, dashed},
    dot_green/.style={circle, fill=cyan!40, inner sep=1.8pt},
    dot_blue/.style={circle, fill=blue!80!black, inner sep=1.8pt},
    prog_arrow/.style={-{Stealth[scale=1.5]}, line width=0.8pt}
  }

  \begin{scope}[shift={(7,1.5)}]
    \draw[surface] (0,0) ellipse (1.5 and 0.8);
    \draw[surface] (-1.8, 0.8) to[out=-53,in=53] (-1.8, -0.8);
    \draw[cross_front] (-0.5, 0.75) arc (90:270:0.25 and 0.75);
    \draw[cross_back] (-0.5, 0.75) arc (90:-90:0.25 and 0.75);
  \end{scope}

  \begin{scope}[shift={(0,1.5)}]
    \draw[surface] (0, 0.8) to[out=-53,in=53] (0, -0.8);
    \node[dot_blue] at (0.28, 0) {};
  \end{scope}

  \begin{scope}[shift={(-3.5,1.5)}]
    \draw[surface, color=white] (0, 0.8) to[out=-53,in=53] (0, -0.8);
    \node[white] at (0.28, 0) {};
  \end{scope}

  \begin{scope}[shift={(2.5, 0)}]
    \draw[thick] (0, 0) -- (0, 3);
    \node[anchor=south] at (0, -0.5) {$\epsilon_0$};
  \end{scope}

    \node[draw, line width=1pt, minimum width=0.9cm, minimum height=0.7cm] at (-2, 2.5) {$\epsilon_-$};

    \node[draw, line width=1pt, minimum width=0.9cm, minimum height=0.7cm] at (4, 2.5) {$\epsilon_+$};

\end{tikzpicture}
\caption{Wall-crossing behaviour when crossing an $\epsilon_0 = 1/m_0$ wall from the $\epsilon_-$ chamber on the left to the $\epsilon_+$ chamber on the right. The dark blue dot is a basepoint of length $m_0$.}
\end{figure}
    
\section{The Moduli Stack of $\epsilon_0$-Semistable Quasimaps to $\bP^N$}\label{S:moduli-stack}

In this section, we give a description of our moduli stack defined by the $\epsilon_0$-semistability condition, and show it is algebraic, of finite type and has affine diagonal. We also describe contraction morphisms which are used throughout the rest of the paper.

\subsection{Description of the Moduli Stack} Given the two stacks $\Qmapplus$ and $\Qmapminus$, we construct a larger algebraic stack $\bigQmap$ that contains both $\Qmapminus$ and $\Qmapplus$ as open substacks, as pictured in the following diagram.

\[
\begin{tikzcd}[row sep=tiny, column sep=tiny]
                                           & \bigQmap &                                             \\
\Qmapplus \arrow[ru, "\subset", phantom, sloped, rotate=-30] &            & \Qmapminus \arrow[lu, "\supset", phantom, sloped, rotate=30]
\end{tikzcd}
\]

As discussed in the previous section, the effect of crossing the wall $\epsilon_0 = 1/m_0$ from left to right replaces length $m_0$ basepoints by degree $m_0$ rational tails. Thus, $\bigQmap$ should include quasimaps with both degree $m_0$ rational tails and length $m_0$ base points. To find objects with infinite automorphism groups, we consider quasimaps that may have a degree $m_0$ rational tail containing a length $m_0$ basepoint. To do this, we impose the following stability condition. 

\begin{defn}\label{D:e0-stability}
    Let $(C, x_1, . . . , x_n, L, \{s_i\}_{i = 0}^N)$ be a prestable quasimap to $\bP^N$. We say this quasimap is \textbf{$\epsilon_0$-semistable} if
    \begin{enumerate}
        \item $\omega_C(x_1+ \ldots + x_n) \otimes L^{\epsilon_0 + \delta}$ is ample, where $0 < \delta << 1$, and 
        \item for every $x \in C$ we have $\ell(x) \leq 1/\epsilon_0$, where $\ell(x)$ is the length as defined in \Cref{D:e-stability}.
    \end{enumerate}

We denote the stack of $n$-pointed $\epsilon_0$-semistable quasimaps of genus $g$ and degree $d$ by $\bigQmap$. When $\epsilon_0$ does not lie on a wall, this definition recovers the notion of $\epsilon$-stability in \Cref{D:e-stability}. A family of prestable quasimaps is $\epsilon_0$-semistable if every geometric fiber is $\epsilon_0$-semistable. If we fix $\epsilon_0 = 1/m_0$, the $\epsilon_0$-semistability condition gives

\begin{itemize}
    \item The degree of $L$ restricted to a rational tail must be $\ge m_0$.
    \item The degree of $L$ restricted to a rational bridge must be strictly positive. This is independent of $\epsilon_0$.
    \item The length of a basepoint must be $\leq m_0$.
\end{itemize}
\end{defn}

In particular, the moduli stack $\bigQmap$ contains \emph{strictly semistable} objects with infinite automorphism groups, namely quasimaps with at least one degree $m_0$ rational tail containing a length $m_0$ basepoint.

\begin{prop}\label{P:alg-stack}
The stack of $\epsilon_0$-semistable quasimaps is an algebraic stack, containing both $\Qmapplus$ and $\Qmapminus$ as open substacks.
\end{prop}
\begin{proof} The stack $\QmapPre$ of prestable quasimaps is algebraic \cite[Section 2.4.1]{cheong2015orbifoldquasimaptheory}. Conditions (1) and (2) of \Cref{D:e0-stability} are open conditions, so $\bigQmap$ is an open substack of $\QmapPre$ and hence algebraic. The stack $\Qmapminus$ is an open substack of $\bigQmap$, cut out by the open condition that there are no rational tails of degree $m_0$. Similarily, $\Qmapplus$ is also an open substack of $\bigQmap$, cut out by the open condition that there are no basepoints of length $m_0$.
\end{proof}

\begin{prop}\label{T:finite-type}
The algebraic stack $\bigQmap$ is of finite type.
\end{prop}
\begin{proof}
    To show that $\bigQmap$ is of finite type, we restrict our attention to rational components as other components will be bounded by the genus. The stack $\bigQmap$ is an open substack of $\QmapPre$, cut out by the following open conditions:

    \begin{enumerate}
        \item There are no unmarked rational tails of degree $< m_0$.
        \item There are no rational bridges of degree 0.
        \item There are no basepoints of length $> m_0$.
    \end{enumerate}

    Conditions (1) and (2) above imply that the number of irreducible rational components is $< k$ for some $k>0$. In particular, the number of rational tails with degree $m_0$ is bounded above by $k$.
    
    Consider the open substack of $\frak{QMap}^{\epsilon_-}_{g, n+k}(\bP^N, d)$ such that when you forget the last $k$ markings, the resulting quasimap lies in $\bigQmap$. This open substack is of finite type since $\frak{QMap}^{\epsilon_-}_{g, n+k}(\bP^N, d)$ is by \Cref{T:e-stableDM}. The morphism defined by forgetting the last $k$ markings defines a surjection onto $\bigQmap$.

    In particular, we have $<k$ unmarked rational tails of degree $m_0$ in $\bigQmap$. In the preimage under the surjective forgetful morphism described above, we can add at least one of the last $k$ markings to each of these tails. This forgetful morphism gives a surjection from a stack of finite type onto $\bigQmap$, showing that $\bigQmap$ is of finite type as well.
\end{proof}

\begin{prop}\label{T:aff-diag}
The algebraic stack $\bigQmap$ has affine diagonal.
\end{prop}
\begin{proof}

    We have a forgetful morphism
    $$f: \bigQmap \xrightarrow{}\frak{M}^{\rm{pol}}$$
    to the stack of marked polarized nodal curves, given by forgetting the sections defining the quasimap to $\bP^N$, and remembering the underlying marked nodal curve with the ample line bundle given by condition (1) of \Cref{D:e0-stability}. The stack $\frak{M}^{\rm{pol}}$ has affine diagonal \cite[Proposition 2.3]{halpernleistner2025quantum}, so it suffices to show the relative diagonal
    \[\Delta_f: \bigQmap \xrightarrow{} \bigQmap \times_{\frak{M}^{\rm{pol}}} \bigQmap\]
    is affine. By \cite[Tag 04YS]{stacks-project}, this follows if $f$ is representable by algebraic spaces and separated.
    
    To show that $f$ is representable by algebraic spaces, we show that for any scheme $T$ and any morphism $T \to \frak{M}^{\rm{pol}}$, the fiber product \[\bigQmap \times_{\frak{M}^{\rm{pol}}} T\] is an algebraic space. The fiber product parametrizes, over any $T' \xrightarrow{} T$, the additional data of $N+1$ sections $s_0, \ldots, s_N \in \Gamma(C_{T'}, L)$ that satisfy the $\epsilon_0$-semistability condition.  We observe that the automorphism group of any such quasimap over a fixed polarized marked curve is trivial. Indeed, suppose we have an automorphism of a quasimap that is the identity on the underlying curve. Automorphisms of the line bundle and sections that fix the underlying curve act by scaling the line bundle. Since the sections are all not identically zero, this will change the sections unless we scale trivially. 

    To show that $f$ is separated, we apply the valuative criterion. Let $R$ be a discrete valuation ring with fraction field $K$, and consider a commutative diagram
\begin{center}
\begin{tikzcd}
\bigQmap \arrow[rr, "f"]                          &  & \frak{M}^{\rm{pol}}                         \\
\Spec(K) \arrow[rr] \arrow[u] &  & \Spec(R) \arrow[u] \arrow[llu, dashed]
\end{tikzcd}
   
\end{center}
where the solid arrows are given. We must show that there exists at most one lift $\Spec(R) \to \bigQmap$ making the diagram commute. Suppose we have two such lifts, corresponding to quasimaps $\mu_1, \mu_2: \mathcal{C} \to [\mathbb{C}^{N+1}/\mathbb{C}^*]$ over $\Spec(R)$ that agree on the generic fiber $\Spec(K)$, for $\mathcal{C} \xrightarrow{} \Spec(R)$ a marked polarized nodal curve. The quasimaps $\mu_1$ and $\mu_2$ are each determined by $N+1$ sections of a line bundle, satisfying the $\epsilon_0$-semistability condition. The total space of the family $\mathcal{C}$ over $\Spec(R)$ is reduced, so sections are uniquely determined by their restriction to a dense open subset. Thus the line bundles and sections defining $\mu_1$ and $\mu_2$, which agree over $\Spec(K)$ must agree everywhere.
\end{proof}

\subsection{Contraction Morphisms for $\epsilon_0$-Semistable Quasimaps to $\bP^N$} \label{S:contractions} In this section we recall the construction of a contraction morphism on rational tails from $\Qmapplus$ to $\Qmapminus$ following \cite[Proposition 2.21]{Toda_2011}. This morphism is also described pointwise in \cite[Section 2.6]{rabano2024contractionmorphismmapsquasimaps} and \cite[Section 3.2.2]{CiocanFontanineKim2014}. The version on families can be found \cite[Theorem 7.1]{Popa_2003} using the language of Quot schemes, or in \cite[Proposition 1.3]{mustata2007intermediate}. These contraction morphisms exist specifically for quasimaps where the target is $\bP^N$. We first give a description of this contraction morphism in families, before extending the contraction morphism to the larger stack $\bigQmap$.

\begin{thm}\cite[Proposition 2.21, Lemma 2.23]{Toda_2011} \label{T: contraction_pm} There is a natural surjective morphism
    $$c_{\bP^N}: \Qmapplus \xrightarrow{} \Qmapminus. $$
\end{thm}

This is called the contraction morphism for $\bP^N$. It can be described pointwise as follows: an $\epsilon_+$-stable quasimap $(C, L, \{s_i\}_{i=0}^N)$ becomes $\epsilon_-$ stable by contracting all degree $m_0$ rational tails, and replacing them with length $m_0$ basepoints. Let   $T_1, \ldots, T_\ell$ denote the rational tails of degree $m_0$ meeting the rest of the curve at a single point, and take
$$\hat{C} := \overline{C \setminus \sqcup_{i=1}^\ell T_i}, \text{ and } \hat{L}:= L|_{\hat{C}} \otimes \cO_{\hat{C}}(\sum_{i=1}^\ell m_0p_i).$$ The sections $\hat{s_j}$ are given as compositions
$$\cO_{\hat{C}} \xrightarrow{s_j} L|_{\hat{C}} \xrightarrow{} L|_{\hat{C}} \otimes \cO_{\hat{C}}(\sum_{i=1}^\ell m_0p_i). $$
The image of the $\epsilon_+$-stable quasimap under the contraction morphism $c_{\bP^N}$ is given by the $\epsilon_-$-stable quasimap $(\hat{C}, \hat{L}, \{\hat{s_i}\}_{i=0}^N)$.\\

We now describe how to make this construction work for families of quasimaps over a scheme $T$. Let $(\pi:\cC \rightarrow T, \{x_i: T \xrightarrow{} \cC\}_{i = 1}^n,\cL, \{s_j\}_{i = 0}^N)$ be an $\epsilon_+$ stable quasimap over a scheme $T$. Since $\cL$ is an effective divisor supported in the smooth locus of $\cC$ by definition, by  \cite[Theorem 3.6]{hassett2002modulispacesweightedpointed}, we get a map
$$f: \cC \xrightarrow{} \Proj_T\left(\bigoplus_{d \geq 0} {\pi_*}(\omega_{\cC}(\sum_i x_i)^{m_0} \otimes \cL)^{dM} \right) =: \hat{\cC}$$
to a flat family of nodal curves $\hat{\cC}$ induced by sections of $(\omega_{\cC}(\sum_i x_i)^{m_0} \otimes \cL_+)^M$ for sufficiently large $M$. This map collapses the locus along which $\omega_{\cC}(\sum_i x_i)^{m_0} \otimes \cL$ has degree zero, and maps its complement birationally onto its image. This locus is precisely the rational tails that have degree $m_0$. Let $\cO_{\cC}(\cT_{m_0})$ denote the line bundle associated to that locus, where $\cT_{m_0}$ is a component of the divisor of rational tails in the universal curve over $\cC$. The line bundle $\cL \otimes \cO_{\cC}(m_0\cT_{m_0})$ has degree zero on the rational tails being contracted, so it descends to a line bundle $\hat{\cL}$ on $\hat{\cC}$, and $N+1$ sections $(\hat{s_0}, \ldots, \hat{s_N})$ are given by the composition

$$f^*\cO_{\hat{\cC}}^{N+1} \xrightarrow{} \cO_{\cC}^{N+1} \xrightarrow{(s_0, \ldots, s_N)} \cL \xrightarrow{} \cL \otimes \cO_{\cC}(m_0\cT_{m_0}).$$

We are left to show that the $\epsilon_-$ stability condition holds, which can be checked pointwise. Let $\cC_t$ be the fibre over a point $t \in T$ that decomposes as $C \cup_{p_i} T_i$ where $f$ contracts the rational tail $T_i$ of degree $m_0$ and maps $C$ birationally onto its image. We have $\coker(s_0, \ldots, s_N)|_C \cong \coker(\hat{s_0}, \ldots, \hat{s_N})|_C \oplus \cO_{p_i}^{m_0}$, where $\cO_{p_i}$ is the skyscraper sheaf at the point $p_i$. Thus after contracting each tail $T_i$ of degree $m_0$, we've added basepoints of length $m_0$ at the points $p_i$, which satisfies $\epsilon_-$-stability.\\

Next, we extend the contraction morphism of \Cref{T: contraction_pm} to the moduli stack of $\epsilon_0$-semistable quasimaps. This morphism also contracts degree $m_0$ rational tails and replaces them with length $m_0$ basepoints.

\begin{thm} \label{T: e_0-contraction}
    There is a natural surjective morphism
    $$\tilde{c}_{\bP^N}: \bigQmap \xrightarrow{} \Qmapminus$$
\end{thm}
\begin{proof} Let $(\pi:\cC_0 \rightarrow T, \{x_i: T \xrightarrow{} \cC_0\}_{i = 1}^n,\cL_0, \{s_j\}_{i = 0}^N)$ be an $\epsilon_0$-semistable quasimap over a scheme $T$. Exactly as described above for $\epsilon_+$ stable quasimaps, by  \cite[Theorem 3.6]{hassett2002modulispacesweightedpointed}, we get a map
$$f: \cC_0 \xrightarrow{} \Proj_T\left(\bigoplus_{d \geq 0} {\pi_*}(\omega_{\cC_0}(\sum_i x_i)^{m_0} \otimes \cL_0)^{dM} \right) =: \hat{\cC}$$
to a flat family of nodal curves $\hat{\cC}$ for sufficiently large $M$. As before, this map collapses the locus along which $\omega_{\cC_0}(\sum_i x_i)^{m_0} \otimes \cL_0$ has degree zero, and maps its complement birationally onto its image. That is, the map $f$ contracts rational tails that have degree $m_0$, where the line bundle $\omega_{\cC_0}(\sum_i x_i^{m_0}) \otimes \cL_0^{\epsilon_0}$ has degree zero. The construction of the line bundle and sections on $\hat{\cC}$ proceeds as described above on families of $\epsilon_+$-stable quasimaps. 

We are left to show that we have constructed an $\epsilon_-$-stable quasimap, which can be done pointwise. We have contracted all rational tails of degree $m_0$; these are exactly the components that become unstable when you move off of the wall $\epsilon_0$ to the $\epsilon_-$ chamber. In addition, the stability condition on basepoint length is the same for $\epsilon_0$-semistability as $\epsilon_-$- stability, so all basepoints of the $\epsilon_0$-semistable family do not become unstable under $\tilde{c}_{\bP^N}$. As described above for the contraction morphism $c_{\bP^N}$, the only new basepoints that we add to the $\epsilon_-$-stable family in the construction above are basepoints of length $m_0$ at the points that were previously attaching nodes to degree $m_0$ rational tails. The argument for surjectivity is the same as for the morphism $c_{\bP^N}$, see for instance \cite[Lemma 2.23]{Toda_2011} or \cite[Theorem 6.0.1]{rabano2024contractionmorphismmapsquasimaps}.
\end{proof}

\section{Good Moduli Space for $\bigQmap$}\label{S:GMS}

In this section, we show that the $\epsilon_0$-semistability condition of \Cref{D:e0-stability} is right notion to guarantee the existence of a proper good moduli space on the wall. 

\begin{thm}\label{thm A}
    $\bigQmap$ admits a proper good moduli space.
\end{thm}
\begin{proof}
    The algebraic stack $\bigQmap$ is of finite type (\Cref{T:finite-type}) with affine diagonal (\Cref{T:aff-diag}). By \cite[Theorem A]{Alper_2023_Existence}, an algebraic stack with these properties admits a separated good moduli space if and only if it is $S$-complete and $\Theta$-reductive. These conditions are established in \Cref{T:S-complete} and \Cref{T:Theta-red}, respectively. The existence part of the valuative criterion for properness is verified in \Cref{T:GMS-proper}.
\end{proof}

\subsection{Preliminaries on Filling Conditions}

We recall the necessary and sufficient filling conditions used in \cite[Theorem A]{Alper_2023_Existence} for an algebraic stack to admit a good moduli space. Throughout, we use $R$ to denote a discrete valuation ring with uniformizer $\pi$ and fraction field $K$.

\begin{defn} \cite[Definitions 3.10, 3.38]{Alper_2023_Existence}
For any morphism of algebraic stacks $f : \fX \to \fY$, we consider a commutative diagram of the form
\[
\xymatrix{
    \cY \setminus 0 \ar[r] \ar[d] & \fX \ar[d] \\
    \cY \ar[r] \ar@{-->}[ur] & \fY
}
\]
for a stack $\cY$ and closed point $0 \in \cY$. We define two conditions on $f$ by requiring that there exists a unique filling of the dotted arrow making the diagram commute for different classes of $\cY$:
\begin{enumerate}
    \item \textbf{$\boldsymbol{\Theta}$-reductivity:} all diagrams in which $\cY = \Spec(R[t]) / \bG_m$ where $0$ corresponds to the ideal $(\pi, t)$, and $\bG_m$ acts by giving $t$ weight $-1$. Define $\Theta_R := \Spec(R[t]) / \bG_m$.
    \item \textbf{$\boldsymbol{S}$-completeness:} all diagrams in which $\cY = \Spec(R[s,t]/(st-\pi)) / \bG_m$ where $0$ corresponds to the ideal $(s, t)$, and $\bG_m$ acts by giving $s$ weight $1$ and $t$ weight $-1$. Define $\STr := \Spec(R[s,t]/(st-\pi)) / \bG_m$.
\end{enumerate}

We denote by $Y$ the stack $\cY$ without the quotient by $\bG_m$, and take $Y^\circ := Y \setminus 0$. We let $(1,0)$ denote the point $(s \neq 0, t=0)$ in $\Spec(R[s,t]/(st-\pi))$ and the point $(\pi \neq 0, t=0)$ in $\Spec(R[t])$. Similarly, let $(0, 1)$ denote the point $(t \neq 0, s=0)$ in $\Spec(R[s,t]/(st-\pi))$ and $(t \neq 0, \pi=0)$ in $\Spec(R[t])$.
\end{defn}

\begin{rem} \label{R:unique-filling}\label{R:adjoin-roots}\label{R:fill-without-Gm-action}Since the stack $\bigQmap$ has affine diagonal by \Cref{T:aff-diag}, we first make the following three observations which simplify the analysis of the filling conditions described above.
\begin{enumerate}
    \item  The filling if it exists for $\Theta$-reductivity is automatically unique by \cite[Proposition 3.17]{Alper_2023_Existence}. Similarly, the filling if it exists for S-completeness is automatically unique by \cite[Proposition 3.41]
    {Alper_2023_Existence}.
    \item We can check S-completeness up to adjoining a root of $\pi$ by \cite[Proposition 3.40]{Alper_2023_Existence}. Similarly, we can check $\Theta$-reductivity up to adjoining a root of $\pi$ by \cite[Proposition 3.17]{Alper_2023_Existence} and a root of $t$ by \cite[Proposition 1.3.11]{halpernleistner2022structureinstabilitymodulitheory}.
    \item It suffices to check $S$-completeness and $\Theta$-reductivity by completing over $Y$ instead of $\cal{Y}$. 
    This follows from the fact that given an algebraic stack with affine diagonal $\frak{X} \xrightarrow{} Y$ with sections $f_1, f_2: Y \to \frak{X}$ that are isomorphic over $Y^\circ$, the section over  $Y^\circ$ induced by a 2-isomorphism $f_1|_{Y^\circ} \xrightarrow{\simeq} f_2|_{Y^\circ}$ extends uniquely to a section of $Y$, since $\rm{Isom}_Y(f_1, f_2) \to Y$ is affine. This reduction is used, for example, in \cite[Lemma A.2]{halpernleistner2023structuremoduligaugedmaps}.
\end{enumerate}
\end{rem}

\begin{lem}\label{L:generically-smooth}
     It suffices to verify $S$-completeness and $\Theta$-reductivity for families of quasimaps over $Y^\circ$ whose generic fibre is a smooth curve.
\end{lem}
\begin{proof}
    We consider each node separately. After a possible finite extension of the discrete valuation ring, every node can be considered as a $K$-point. We take the closure of this point in the total space of $C^\circ/\bG_m$, where $C^\circ \xrightarrow{} Y^\circ$ is the family of curves, giving a section. The original family of quasimaps is glued from two families of curves, each with an additional marking at the former node. After analyzing the extension problem independently for each such family with additional markings, we may identify components along these markings to obtain a filling for the original family over $Y$.
    
    In particular, we do not need to consider marked points that develop into rational chains over $(1,0)$ and $(0,1)$. There are two cases for the generic fibre over $Y^\circ$: either it is (1) a degree $m_0$ rational tail with one marking, or it is (2) $\epsilon_-$-stable. For (1), rational tails of degree $<m_0$ are not allowed, so the $\bP^1$ cannot degenerate further into irreducible components. For (2), we know that $(0,1)$ and $(1,0)$ must have the same underlying $\epsilon_-$-stable model. This follows since we can always apply the contraction morphism $\tilde{c}_{\bP^n}$ which contracts degree $m_0$ rational tails over $(0,1) $ and $(1,0)$, and since the generic fibre was $\epsilon_-$-stable, the fibres at $(0,1) $ and $(1,0)$ must be the same by separatedness of $\Qmapminus$. That is, if we consider a generically $\epsilon_-$-stable quasimap degenerating to an $\epsilon_0$- semistable quasimap in the fibres over  $(1,0)$ and $(0,1)$, by the explicit description of $\tilde{c}_{\bP^n}$ and separatedness of $\Qmapminus$, we must recover the fibres over $(0,1) $ and $(1,0)$ of the original family of quasimaps over $Y^\circ$, except for the addition of degree $m_0$ rational tails.
\end{proof}

\noindent When working component by component, there are two main cases to consider.

\paragraph{Case I: a Single Rational Tail.}
We first discuss the case where the component is a single rational tail of degree $m_0$ in the generic fibre, where $(g, n, d) = (0, 1, m_0)$. Since rational
tails of degree $< m_0$ are not allowed, the curve cannot further degenerate into irreducible components.

\begin{prop}\label{P:degenerate-case}
    Fix $(g, n, d) = (0, 1, m_0)$. Any family of $\epsilon_0$-semistable quasimaps to $\bP^N$ over $Y^\circ$ whose generic fibre is single rational tail with degree $m_0$ extends uniquely to a family over $Y$.
\end{prop}

\begin{proof}

If the family of quasimaps over $Y^\circ$ contains no length $m_0$ basepoint, it is  $\epsilon_+$-stable and the map $Y^\circ \to \bigQmap$ factors through the open substack $\Qmapplus$. Since $\Qmapplus$ is a separated DM stack, the extension over $Y$ exists and is unique. Thus, we focus on the case where we have strictly semistable objects. We start with an $\epsilon_0$-semistable quasimap $(\cC^\circ, \cL^\circ, \{s_i^\circ\}_{i=0}^N)$ over $Y^\circ$, where the generic fibre is a single rational tail of degree $m_0$ containing a length $m_0$ basepoint.

First, we fill the family of curves $\cC^\circ \xrightarrow{} Y^\circ$ to a family of curves $\cC \xrightarrow{} Y$. Since basepoints of the quasimap are disjoint from nodes, we may regard our rational tail as a $\bP^1$ with two distinguished points: the given marking and a basepoint. Considering the basepoint as a second marking, we have a family of $\bP^1$'s with two marked points. The moduli stack $\mathcal{M}_{0,2}$ of two-pointed rational curves is isomorphic to $B\bC^*$. As $B\bC^*$ is both $\Theta$-reductive and $S$-complete, the family of curves over $Y^\circ$ extends uniquely to $Y$. More specifically, the unique extension of the family of curves is $\cC := \bP_{Y}(\mathcal{N} \oplus \mathcal{O}_{Y})$ for some line bundle $\mathcal{N}$ on $Y$.

Next, we extend the line bundle $\cL^\circ \xrightarrow{} \cC^\circ$ to a line bundle $\cL \xrightarrow{} \cC$. The restriction map $\Pic(\bP_{Y}(\mathcal{N} \oplus \mathcal{O}_{Y})) \xrightarrow{} \Pic(\cC^\circ)$ is an isomorphism since $\Pic(Y^\circ) \cong \Pic(Y)$. Thus any line bundle $\cL^\circ$ extends uniquely to a line bundle $\cL := \iota_*\cL^\circ $ where $\iota: \cC^\circ \hookrightarrow \cC$ denotes the inclusion.

Lastly, we show that the sections $s_0^\circ, \ldots, s_N^\circ$ of $\cL^\circ$ giving the quasimap extend uniquely to $(s_0, \ldots, s_N) \in \Gamma(\cC, \cL)$ in a way that is $\epsilon_0$-semistable. In particular, we need to show that the basepoint is disjoint from the marked point (the node) in the fibre over 0. Consider the marking as corresponding to a section $\sigma: Y \rightarrow \bP_{Y}(\mathcal{N} \oplus \mathcal{O}_{Y})$. The condition that the basepoint is disjoint from the node over $Y^\circ$ means that not all sections $s_i$ vanish along $\sigma(Y^\circ)$. Since the vanishing locus of a section along $\sigma(Y)$ is a divisor, it cannot vanish only at a codimension 2 point. Therefore, if some section $s_i$ is nonvanishing along the image of $\sigma$ over both $(1,0)$ and $(0,1)$, then it remains nonvanishing at $\sigma(0)$. If no single section is nonvanishing along $\sigma$ on both $(1,0)$ and $(0,1)$, we may take a linear combination of sections that are nonvanishing on each of $(1,0)$ and $(0,1)$. Such sections exist because the original family over $Y^\circ$ is a quasimap, so over $(1,0)$ and $(0,1)$, there is at least one section not vanishing at the node. Their sum then does not vanish along $\sigma$ over both $(1,0)$ and $(0,1)$, hence also not at $0$. \end{proof}

\paragraph{Case II: Creating Rational Tails.}\label{general-case-2} We are left to address the main case, when $(g, n, d) \neq (0, 1, m_0)$ and there are no degree $m_0$ rational tails, so the generic fibre is $\epsilon_-$-stable. This will be the focus of the following three subsections of the paper. In this case, extensions over $Y$ are constructed by introducing rational tails of degree $m_0$ where necessary.

\subsection{A Local $\bA^1$ Model for the Filling}

When the generic fibre is $\epsilon_-$-stable, as in Case II above, we show that constructing a filling over $Y$ can be reduced to constructing a filling in formal neighbourhoods around a finite  number of points on the source curve. We begin with the data $(\cC, \cL, \mu)$ consisting of a $\bG_m$-equivariant family of curves $\cC$ over $Y^\circ$ with an $\epsilon_0$-semistable quasimap $\mu: \cC \xrightarrow{} \bP^N$ given by $N+1$ sections of the line bundle $\cL$. By \Cref{R:fill-without-Gm-action}, it is enough to fill both the family of curves $\cC$ and the quasimap $\mu$ over $Y$ as opposed to the quotient stack $\mathcal{Y}$. We begin by providing a summary of the argument to reduce the filling to a local problem, and explain each step in detail afterwards.

\begin{enumerate}
    \item Since $Y^\circ$ admits a morphism to $\Spec(R)$, we utilize the contraction morphism from \Cref{T: e_0-contraction} to first construct an underlying $\epsilon_-$ stable family $(\cC_-, \cL_-, \mu_-)$ of quasimaps over $\Spec(R)$. From \Cref{L:open-dense}, we get a non-empty fiberwise dense open $\cC_-^\circ \subset \cC_-$ agreeing with $\cC$ except for finitely many points $p$ in the special fibre of the family that were contracted. The quasimaps $\mu, \mu_-$ also agree on dense opens as in diagram \ref{eqn:dense-open-commute}.

    \item We extend the family of curves $\cC$ over $Y$ by realizing it as a blow-up of the underlying $\epsilon_-$ stable family $(\cC_-)_Y$ along a closed subscheme $\tilde{Z}$ in \Cref{lem:filling-is-a-blowup}. We basechange along $\Spec(\widehat{\cO_{\cC_-, p}}) \rightarrow \cC_-$ and get a pushout of stacks in \Cref{L: first_pushout}. We transfer the blowups of the family $(\cC_-)_Y$ to blowups on $\Spec(\widehat{\cO_{\cC_-, p}})_{Y}$. By \Cref{lem: affine-blowup}, this is analogous to blowing up $\Spec(\widehat{\cO_{\bA^1_R, 0}})_{Y}$ at an ideal $Z$ whose supported is contained in the union of the two axis' in Y.

    \item Lastly, we extend the map $\mu$ to $\tilde{\mu}: Bl_{\tilde{Z}}((\cC_-)_{Y}) \xrightarrow{} \bP^N$, agreeing with $\mu$ over $Y^\circ$ as in diagram \ref{eqn:maps-commute}. We again basechange along $\Spec(\widehat{\cO_{\cC_-, p}}) \rightarrow \cC_-$ and get a pushout of stacks in \Cref{lem:blowup_completion}, along with an equivalence of categories allowing us to fill the map by providing a map from the formal neighborhood of the points $p$.
\end{enumerate}
\paragraph{Step 1.} We construct the underlying $\epsilon_-$-stable family $\cC_-$ along with a fiberwise dense open $\cC_-^\circ \subset \cC$.

\begin{lem} \label{L:open-dense} Let $(\cC_1,\cL_1, \mu_1)$, and $(\cC_2,\cL_2, \mu_2)$ be two families of $\epsilon_0$-semistable quasimaps over $\Spec(R)$ that are generically $\epsilon_-$-stable and agree over $\Spec(K)$. Let $(\cC_-,\cL_-, \mu_-)$ denote the common image of both $\cC_1$ and $\cC_2$ under the contraction morphism $\tilde{c}_{\bP^N}$ of \Cref{T: e_0-contraction}. There exists a non-empty fiberwise open dense subset $\cC_-^\circ$ such that the restriction $(\tilde{c}_{\bP^N}^{(i)})^{-1}(\cC_-^\circ) \to \cC_-^\circ$ is an isomorphism for $i=1,2$.
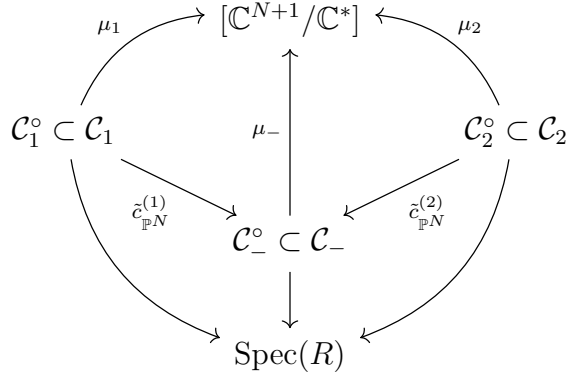
\begin{center}
\begin{tikzcd}
                                                                                                    & {[\bC^{N+1}/\bC^*]}                                       &                                                                                                     \\
\cC_1^\circ \subset \cC_1 \arrow[ru, "\mu_1", bend left] \arrow[rd, "\tilde{c}_{\bP^N}^{(1)}"'] \arrow[rdd, bend right] &                                                           & \cC_2^\circ \subset \cC_2 \arrow[lu, "\mu_2"', bend right] \arrow[ld, "\tilde{c}_{\bP^N}^{(2)}"] \arrow[ldd, bend left] \\
                                                                                                    & \cC_-^\circ \subset \cC_- \arrow[uu, "\mu_-"] \arrow[d] &                                                                                                     \\
                                                                                                    & \Spec(R)                                                  &                                                                                                    
\end{tikzcd} 
\end{center}
Moreover, the morphisms $\mu_1|_{\cC_1^\circ}, \mu_2|_{\cC_2^\circ}$ and $\mu_-|_{\cC_-^\circ}$ all agree and the total spaces $\cC_1, \cC_2$ and $\cC_-$ are birational.
\end{lem}

\noindent The above diagram does not commute away from $\cC_1^\circ, \cC_2^\circ$ and $\cC_-^\circ$.
\begin{proof}
Let $\tilde{c}_{\bP^N}^{(i)}: \cC_i \to \cC_-$ denote the contraction morphism from \Cref{T: e_0-contraction}. The two families agree over $\Spec(K)$, so by separatedness of $\Qmapminus$, we may take $\cC_-$ as a common target.

Define $\cC_-^\circ \subset \cC_-$ to be the maximal open subset $\cal{U}$ of $\cC_-$ such that both morphisms $(\tilde{c}_{\bP^N}^{(i)})^{-1}(\cal{U}) \to \cal{U}$ are isomorphisms. We have that $\cC_-^\circ$ is non empty and fiberwise dense from the fiberwise description of the contraction morphism in \Cref{T: e_0-contraction}, since we are only removing finitely many points in the central fibre of $\cC_-$ (namely, the images of the nodes where contracted rational tails were attached). The exceptional locus of $\tilde{c}_{\bP^N}^{(i)}$ consists of all rational tails of degree $m_0$ that are contracted; denote this locus by $E^{(i)} \subset \cC_i$. The images $\tilde{c}_{\bP^N}^{(i)}(E^{(i)})$ are finite sets of points in each fibre of $\cC_-$. Define the dense open subsets $\cC_i^\circ \subset \cC_i$ as the preimages of $\cC_-^\circ$ under these contractions.

We show that $\mu_1|_{\cC_1^\circ}$, $\mu_2|_{\cC_2^\circ}$, and $\mu_-|_{\cC_-^\circ}$ all agree under the construction in \Cref{T: e_0-contraction}, which gives a relation between the sections defining $\mu_i$ and $\mu_-$. The contracted quasimap $\mu_-$ on $\cC_-$ is given by a line bundle $\mathcal{L}_-$ and its sections. This line bundle is constructed as 

$$\cL_- \cong (\tilde{c}_{\bP^N}^{(i)})_* (\cL_i \otimes \cO_{\cC_i}(m_0\mathcal{T}_{m_0}))$$

where $\mathcal{T}_{m_0}$ is the divisor of contracted tails having degree $m_0$. The quasimaps $\mu_i$ are given by sections $(s_{i,0}, \dots, s_{i,N})$ of $\cL_i$, while the contracted quasimap $\mu_-$ is given by sections $(t_0, \dots, t_N)$ of $\cL_-$. The construction identifies that over the locus where $\tilde{c}_{\bP^N}$ is an isomorphism, we have:
$$(\tilde{c}_{\bP^N}^{(i)})^*(t_j) = \sigma \cdot s_{i, j}$$
where $\sigma$ is a canonical section of $\cO_{\mathcal{C}_i}(m_0 \mathcal{T}_{m_0})$. On the open set $\cC_i^\circ$, we have removed the support of $\mathcal{T}_{m_0}$, so the section $\sigma$ is nowhere vanishing. Consequently, the map
$$\cL_i|_{\cC_i^\circ} \xrightarrow{\otimes \sigma} (\cL_i \otimes \cO_{\cC_i}(m_0 \mathcal{T}_{m_0}))|_{\cC_i^\circ} \cong ((\tilde{c}_{\bP^N}^{(i)})^*(\cL_-))|_{\cC_i^\circ}$$
is an isomorphism. Under this isomorphism, the sections $s_{i,j}|_{\cC_i^\circ}$ defining $\mu_i|_{\cC_i^\circ}$ correspond to the pullbacks $(\tilde{c}_{\bP^N}^{(i)})^*(t_j)|_{\cC_i^\circ}$ of the sections defining $\mu_-|_{\cC_-^\circ}$. Hence, $\mu_i|_{\cC_i^\circ} = \mu_-|_{\cC_-^\circ} \circ \tilde{c}_{\mathbb{P}^N}^{(i)}|_{\cC_i^\circ}$. Since $\tilde{c}_{\mathbb{P}^N}^{(i)}|_{\cC_i^\circ}$ is an isomorphism onto $\cC_-^\circ$, it follows that $\mu_1|_{\cC_1^\circ}$ and $\mu_2|_{\cC_2^\circ}$ are both identified with $\mu_-|_{\cC_-^\circ}$ via the respective contraction morphisms.

Birationality of $\cC_1, \cC_2$ and $\cC_-$ follows since the generic fibres of $\cC_1$, $\cC_2$, and $\cC_{-}$ are isomorphic since no rational tails of degree $m_0$ occur generically, as the generic fibre is $\epsilon_-$-stable by hypothesis.\end{proof}

The moduli stack $\Qmapminus$ is a separated DM stack, hence any morphism $Y^\circ \xrightarrow{} \Qmapminus$ factors through $\Spec(R)$. This factorization together with the contraction morphism $\tilde{c}_{\bP^N}$ yields an underying $\epsilon_-$-stable family of curves $\cC_-$ over $\Spec(R)$ along with a morphism $\mu_-: \cC_- \xrightarrow{} \bP^N$. By \Cref{L:open-dense}, we obtain fiberwise dense opens $\cC_-^\circ \subset \cC_-$  and $\cC^\circ \subset \cC$ where $\cC^\circ := \tilde{c}_{\bP^N}^{-1}\left((\cC_-)_{Y^\circ}\right)$ such that the following diagram commutes:

\begin{equation}\label{eqn:dense-open-commute}
\centering
\begin{tikzcd}
\cC^\circ \arrow[rr, "\mu"] \arrow[d, "\tilde{c}_{\bP^N}"] &  & {[\bC^{N+1}/\bC^*]} \\
(\cC_-^\circ)_{Y^\circ} \arrow[rru, "\mu_-"']               &  &                    
\end{tikzcd}
\end{equation}

\paragraph{Step 2.} Next we extend the family of curves $\cC$ over $Y$, by realizing it as a blow-up of the underlying $\epsilon_-$-stable family.

\begin{prop}\label{lem:filling-is-a-blowup}
    The family of curves $\cC$ corresponding to an $\epsilon_0$-semistable quasimap over $\Spec(R)$ that is generically $\epsilon_-$-stable is obtained from the underlying $\epsilon_-$-stable family $\cC_-$ by blowing up a closed subscheme.
\end{prop}
    \begin{proof}
        The contraction morphism $\tilde{c}_{\bP^N}$ is birational on the total spaces of the families of curves since the $\epsilon_0$-semistable family agrees with the underlying $\epsilon_-$-stable family over the generic fibre by construction. The morphism is induced by the relative Proj of the graded sheaf
        $$\bigoplus_{d \ge 0} \pi_* (\omega_{\cC}(\sum_i x_i) \otimes \mathcal{L}^{\epsilon_0})^{dM} $$
        for sufficiently large $M$, as in the construction of the contraction. Thus $\tilde{c}_{\bP^N}$ is a projective morphism between integral schemes, as we are working in the case of a smooth generic fibre. By \cite[Theorem 8.1.24]{liu2002algebraic}, a projective birational morphism of integral schemes over a Noetherian affine scheme is a blowing-up morphism along a closed subscheme.
    \end{proof}  

The filling can be realized as a blowup of the underlying $\epsilon_-$-stable family by the above lemma, so our goal is to exhibit a ideal corresponding to closed subscheme $\tilde{Z} \hookrightarrow (\cC_-)_{Y}$ such that
    $$\Bl_{\tilde{Z}}((\cC_-)_{Y})|_{Y^
    \circ} \cong \cC$$
where $\cC$ is the family of curves we started with over $Y^\circ$. In order to produce such an ideal corresponding to $\tilde{Z}$, we can work locally using the following lemma.

\begin{lem}\label{L: first_pushout}
 The following square is a pushout in the category of quasi-compact algebraic stacks with affine diagonal

\begin{center}
\begin{tikzcd}
{(\Spec(\widehat{\cO_{\cC_-, p}})\backslash p)_{Y}} \arrow[d, hook] \arrow[rr] &  & (\cC_-\backslash p)_{Y} \arrow[d, hook] \\
{(\Spec(\widehat{\cO_{\cC_-, p}}))_{Y}} \arrow[rr, "f"]                                                                    &  & (\cC_-)_{Y}                     

\end{tikzcd}   
\end{center}
where $p$ is a point in the special fibre of $\cC_-$ over $\Spec(R)$ that is not contained in $\cC_-^\circ$.
\end{lem}
\begin{proof}
    This is the basechange of a pushout square over $\Spec(R)$. We show that it is a pushout over $\Spec(R)$ by applying \cite[Lemma 3.3.3, Corollary 3.3.7]{halpernleistner2019mappingstackscategoricalnotions}. In the notation of \cite[Lemma 3.3.3]{halpernleistner2019mappingstackscategoricalnotions}, $\cal{X} = \cC_-$, $Z = \{p\}$ and $\mathcal{Y} = \mathcal{U} = \cC_- \setminus p$. To show that the conditions of Lemma 3.3.3 in \cite{halpernleistner2019mappingstackscategoricalnotions} are satisfied, we need to show that $\Spec(\widehat{\cO_{\cC_-, p}})$ and $\cC_-$ have the same formal completion along $f^{-1}(p)$ and $p$ respectively. The preimage under the morphism $f$ of $p$ is the maximal ideal $\widehat{m}_p$ in the completed local ring $\widehat{\cO_{\cC_-, p}}$ corresponding to the maximal ideal $m_p \in \cO_{\cC_-, p}$. The formal completion of $\cC_-$ along $p$ is 
    $$\Spf(\lim_{\leftarrow n} \cO_{\cC_-, p}/m_p^n ) \cong \Spf(\widehat{\cO_{\cC_-, p}})$$
    which is the same as the formal completion of $\Spec(\widehat{\cO_{\cC_-, p}})$ along $\widehat{m}_p$, since $\Spec(\widehat{\cO_{\cC_-, p}})$ is already completed with respect to $\widehat{m}_p$. The map $f^{-1}(p) \to p$ is a morphism of schemes 
    $i: \Spec(\kappa(\widehat{\mathfrak{m}}_p)) \to \Spec(\kappa(p))$. 
    Since completion of a local ring does not change the residue field, we have 
    $\kappa(\widehat{\mathfrak{m}}_p) \cong \kappa(p)$ canonically. 
    Therefore $i$ is an isomorphism, and in particular flat, satisfying the flatness hypothesis 
    of Corollary 3.3.7 of \cite{halpernleistner2019mappingstackscategoricalnotions}.
    
\end{proof}

The formally étale morphism $f:(\Spec(\widehat{\cO_{\cC_-, p}}))_{Y} \rightarrow(\cC_-)_{Y}$ above is isomorphic along a closed subset, so we get a bijection of ideals that are set theoretically supported on the union of the two axis' in $Y$ over the point $p$. Thus, in order to extend the family of curves $\cC$ over $Y$ via blowup, we can work locally and blowup at an ideal in $(\Spec(\widehat{\cO_{\cC_-, p}}))_{Y}$.

\begin{lem} \label{lem: affine-blowup} The following is an isomorphism over $Y$
    $$\Spec(\widehat{\cO_{\cC_-, p}})_{Y} \cong \Spec(\widehat{\cO_{\bA^1_R, 0}})_{Y}$$
    where $p$ is a point in the special fibre of $\cC_- \xrightarrow{} \Spec(R)$ that is not contained in the dense open $\cC_-^\circ$.
\end{lem}
\begin{proof} By the construction of the contraction morphism in \Cref{T: e_0-contraction}, $\cC_{-}$ is smooth at $p$. Thus, $\widehat{\cO_{\cC_-, p}} \cong R[[x]]$ where $x$ is the local coordinate around the point $p$.
\end{proof}

\paragraph{Step 3.} Once we have constructed the family of curves via blowup, we must then construct an extension of the quasimap $\tilde{\mu}: \Bl_{\tilde{Z}}((\cC_-)_{Y}) \xrightarrow{} [\bC^{N+1}/\bC^*]$ such that the following diagram commutes:
\begin{equation}\label{eqn:maps-commute}
\centering
\begin{tikzcd}
\left(\Bl_{\tilde{Z}}((\cC_-)_{Y})\right)|_{Y^\circ} \arrow[d, "\cong"] \arrow[rr, "\tilde{\mu}"] &  & {[\bC^{N+1}/\bC^*]} \\
\cC \arrow[rru, "\mu"]                                                       &  &                    
\end{tikzcd}
\end{equation}
Similarly to extending the curve, we can extend the quasimap locally by basechanging along the morphism $\Spec(\widehat{\cO_{\cC_-, p}}) \xrightarrow{} \cC_-$.
\begin{lem}\label{lem:blowup_completion}
    The following square is a pushout in the category of quasi-compact algebraic stacks with affine diagonal
    \[
    \begin{tikzcd}
    (\Spec(\widehat{\cO_{\cC_-, p}}) \setminus p)_{Y} \arrow[d, hook] \arrow[r] & (\cC_- \setminus p)_{Y} \arrow[d, hook] \\
    \Bl_Z((\Spec(\widehat{\cO_{\cC_-, p}}))_{Y}) \arrow[r, "g"] & \Bl_{\tilde{Z}}((\cC_-)_{Y})
    \end{tikzcd}
    \]
    where $p$ is a point in the special fibre of $\cC_- \xrightarrow{} \Spec(R)$ that is not contained in $\cC_-^\circ$, and $\tilde{Z} \subset (\cC_-)_{Y}$ is a closed subscheme set-theoretically contained in the preimage of $p$, with  $Z \subset (\Spec(\widehat{\cO_{\cC_-, p}}))_{Y}$ its preimage under the
    map $f$ of \Cref{L: first_pushout}. Moreover,
    $$\Qcoh(\Bl_{\tilde{Z}}((\cC_-)_{Y})) \rightarrow \Qcoh(\Bl_Z((\Spec(\widehat{\cO_{\cC_-, p}}))_{Y}))\times_{\Qcoh((\Spec(\widehat{\mathcal{O}_{\cC_-,p}}\setminus p))_{Y})} \Qcoh((\cC_-\backslash p)_{Y})$$
    is an equivalence of categories.
\end{lem}

\begin{proof} As in the proof of \Cref{L: first_pushout}, we apply \cite[Corollary 3.3.7]{halpernleistner2019mappingstackscategoricalnotions}. We need to show that $g$ induces an isomorphism on formal completions. Consider the commutative diagram:
    
    \[
    \begin{tikzcd}
    (\Spec(\widehat{\cO_{\cC_-, p}}) \setminus p)_{Y} \arrow[d, hook] \arrow[r] & 
    (\cC_- \setminus p)_{Y} \arrow[d, hook] \\
    \Bl_Z((\Spec(\widehat{\cO_{\cC_-, p}}))_{Y}) \arrow[d, "\pi_Z"] \arrow[r, "g"] & 
    \Bl_{\tilde{Z}}((\cC_-)_{Y}) \arrow[d, "\pi_{\tilde{Z}}"] \\
    (\Spec(\widehat{\cO_{\cC_-, p}}))_{Y} \arrow[r, "f"] & 
    (\cC_-)_{Y}
    \end{tikzcd}
    \]

    The morphism $f$ is flat because it is a base change of the completion map $\Spec(\widehat{\cO_{\cC-,p}} )\to \cC_-$, which is flat. By construction $\tilde{Z}$ is the pullback of $Z$ under $f$. Hence, by the flat base change property of blow‑ups \cite[Tag 0805]{stacks-project}, the lower square is Cartesian.
    Therefore, $f$ inducing an isomorphism on formal completions along $p$, implies that so does $g$ along the preimage of $p$, namely $Z$ and $\tilde{Z}$. Thus the conditions of \cite[Lemma 3.3.3]{halpernleistner2019mappingstackscategoricalnotions} are satisified, giving the pushout above. The resulting equivalence of categories induced by restriction follows directly from \cite[Lemma 3.3.3]{halpernleistner2019mappingstackscategoricalnotions}. 
\end{proof}

The above lemma says that maps into $[\bC^{N+1}/\bC^*]$
can also be reduced to a local extension problem. A quasimap from $\Bl_{\tilde{Z}}((\cC_-)_{Y})$ to \([\mathbb{C}^{N+1}/\mathbb{C}^*]\) is given by the data of a morphism $\cO^{N+1} \xrightarrow{} \cL$ on both $\Bl_Z((\Spec(\widehat{\cO_{\cC_-, p}}))_{Y})$ and $(\cC_-\backslash p)_Y$, together with an isomorphism on $(\Spec(\widehat{\cO_{\cC_-, p}}) \setminus p)_{Y}$ identifying one family of sections with the other. This gives a morphism of line bundles $\cO^{N+1} \xrightarrow{} \cL$ on $\Bl_{\tilde{Z}}((\cC_-)_{Y})$.

\paragraph{Next Steps.} We are left to find an ideal corresponding to a closed subscheme $Z$ whose support is contained in the union of the two axis' in $Y$ such that our original family of curves $\cC$ is recovered via a blowup, and our original map $\mu$ extends to a $\epsilon_0$-semistable quasimap $\Bl_Z((\Spec(\widehat{\cO_{\bA^1_R, 0}}))_Y) \xrightarrow{} \bP^N$. In subsequent sections, by analyzing all configurations of $\epsilon_0$-semistable quasimaps over $Y^\circ$, we will produce such ideals $Z$.

\subsection{\texorpdfstring{$S$}{S}-Completeness}\label{sec:S-complete}

In this section, we establish $S$-completeness for $\bigQmap$ by analyzing the local filling problem described in the above subsection.

\begin{thm} \label{T:S-complete} The algebraic stack $\bigQmap$ is S-complete.
\end{thm}

We reduce to the case of filling over $\Spec(R[s, t]/(st-\pi))$ by \Cref{R:fill-without-Gm-action}, and uniqueness of the filling also follows from \Cref{R:unique-filling}. We analyze each component of the source curve separately as discussed in \Cref{L:generically-smooth}. The case where the generic fibre is a single spinning rational tail of degree $m_0$ with a length $m_0$ basepoint (Case I) is is \Cref{P:degenerate-case}. We are left to produce a filling for the case where the generic fibre is $\epsilon_-$-stable (Case II). To show $S$-completeness holds in this situation, we analyze the local $\bA^1$-filling problem discussed in the above subsection for all possible configurations of $\epsilon_0$-semistable quasimaps over $\Spec(R[s, t]/(st-\pi))\setminus 0$. Throughout the analysis, we freely adjoin roots of the uniformizer $\pi$, which is allowed by \Cref{R:adjoin-roots}.

\paragraph{Case Analysis for Local Fillings.}
A map from $\Spec(R[s, t]/(st-\pi)) \setminus 0$ to the stack $\bigQmap$ corresponds to two families of $\epsilon_0$-semistable quasimaps over $\Spec(R)$ that agree over the generic point $\Spec(K)$. Each family has its own special fibre over the closed point of $\Spec(R)$; we refer to these as the two special fibres below.  In the rest of this section, we analyze all possible configurations of these two special fibres and construct in each case an extension over the codimension two point $s=t=0$.

\paragraph{Case 1: $\epsilon_-$-stable on both $(0,1)$ and $(1,0)$, or $\epsilon_+$-stable on both $(0,1)$ and $(1,0)$}\label{C:trivial-case} We  consider the case where the special fibres over both copies of $\Spec(R)$ are strictly $\epsilon_-$-stable, and the $\epsilon_+$-stable case is analogous. Openness of $\epsilon_-$-stability implies that the map from the generic point $\Spec(K)$ is $\epsilon_-$-stable as well. The $\epsilon_-$-stable locus is separated by \Cref{T:e-stableDM}, so we get unique filling over $s = t = 0$ that is $\epsilon_-$-stable.

\paragraph{Case 2: Blowing up at a strict transform of one axis.} Next, we analyze the situation where we blowup at a strict transform of either the $s$-axis or the $t$-axis. Without loss of generality, we outline the construction for a blowup at a strict transform of the $t$-axis, and the $s$-axis is analogous. For this case, note that both a family of $\epsilon_+$-stable and $\epsilon_0$-semistable quasimaps have an underlying $\epsilon_-$-stable family via the contraction morphisms in \Cref{T: contraction_pm} and \Cref{T: e_0-contraction}, respectively.

\subparagraph{Case 2a: $\epsilon_+$-stable over $(1,0)$ and $\epsilon_-$-stable over $(0,1)$.} Consider the situation where over one axis, we get the $\epsilon_-$-stable limit where a length $m_0$ basepoint develops. On the other, we get the $\epsilon_+$-stable limit with a degree $m_0$ rational tail. Over $\Spec(K)$, we have both $\epsilon_-$ and $\epsilon_+$-stability.

\begin{prop}
\label{T:epsilon-pm}
    The quasimap to $\bP^N$ from the family of nodal curves given by $\mathrm{Bl}_{(x, t^b)}(\bA_x^1 \times \STr)$ for certain $b \in \bQ_{>0}$ is $\epsilon_-$-stable over (0,1), $\epsilon_+$-stable over (1,0), and is $\epsilon_0$-semistable over $s = t = 0$, after passing to a ramified cover.
\end{prop}
\begin{proof}
We start with the data of an $\epsilon_-$-stable family over the copy of $\Spec(R)$ where $t = 1, s = \pi$ given below:

\begin{itemize}
    \item $\cC_1 = \bA^1_R$
    \item $\cL_1 = \cO_{\bA^1_R}$
    \item Sections $\varphi_i(x, \pi) = x^{m_0} \alpha_i(x) + \pi\beta_i(x, \pi)$ for $0 \le i \le N$ that are polynomials in $x$ and $\pi$, developing a length $m_0$ basepoint at $s=\pi = 0, x = 0$. Up to a linear change of coordinates, we can consider $\alpha_0(0) = 1$.
\end{itemize}

Let $E \cong \bP^1$ be the exceptional divisor of the blowup at the ideal $(x, t^b)$ with coordinates $[u :v]$. The blowup has equation $xu = t^bv$. We can describe the two charts of the blowup as follows:

\begin{itemize}
    \item$U_1$ is the locus where $v \neq 0$, with coordinates $x, s, \frac{u}{v}$. We have $t^b = x \cdot \frac{u}{v}$. This chart contains the node.
    \item$U_2$ is the locus where $u \neq 0$, with coordinates $s, t, \frac{v}{u}$. We have $x = t^b \cdot \frac{v}{u}$.
\end{itemize}

Over the copy of $\Spec(R)$ where $s = 1,t = \pi$, we construct the unique $\epsilon_+$ stable limit, by developing a degree $m_0$ rational tail via a blowup. The line bundle giving the quasimap is $\cO(-m_0E)$. We describe $N+1$ sections $\psi_i$ of $\cO(-m_0E)$.

The blowup gives a canonical section $f: \cO \xrightarrow{} \cO(E)$, trivialized away from $E$. Over the locus where $s \neq 0, t \neq 0$, the sections $\varphi_i$ and $\psi_i$ agree via $f$. We can describe this canonical section $f$ chartwise:

\begin{itemize}

    \item $\frac{x}{v}$ in the chart $U_1$ 
    \item $\frac{t^b}{u}$ in the chart $U_2$
\end{itemize}

We now construct the desired sections $\psi_i$ of $\cO(-m_0E)$. First, we pass to a ramified cover, adjoining $\sqrt[c]{\pi}$ for $c \in \bN$. This allows us to make sense of $b$ as a rational number where the denominator divides $c$. In the $U_1$ chart containing the node, our sections never vanish since basepoints are disjoint from nodes of the quasimap, so we focus on the chart $U_2$. We have:
\begin{align*}
    \psi_i\left(t^b \cdot \frac{v}{u}, \pi \right) \cdot f^{-m_0} &= \left[t^{bm_0}\left(\frac{v}{u}\right)^{m_0} \alpha_i\left(t^b \cdot \frac{v}{u}\right) + st \beta_i\left(t^b \cdot \frac{v}{u}, st\right)\right] \cdot \left(\frac{t^b}{u}\right)^{-m_0}\\
    &= \left[\left(\frac{v}{u}\right)^{m_0} \alpha_i\left(t^b \cdot \frac{v}{u}\right) + st^{1-bm_0} \beta_i\left(t^b \cdot \frac{v}{u}, st\right)\right] \cdot u^{m_0}
\end{align*}

We want to ensure we have an $\epsilon_+$-stable limit where $s = 1,t = \pi$. We focus on the $\beta_i$ term in order to ensure the correct basepoint length.
\begin{align*}
st^{1-bm_0} \beta_i\left(t^b \cdot \frac{v}{u}, st\right) &= st^{1-bm_0} \sum_{j_i, k_i \ge 0} a_{j_i,k_i}\left(t^b \cdot \frac{v}{u}\right)^{j_i}(st)^{k_i}\\
&= \sum_{j_i, k_i \ge 0} a_{j_i,k_i}\left(\frac{v}{u}\right)^{j_i}s^{(k_i + 1)}t^{(k_i + 1) + b(j_i - m_0)}
\end{align*}
To get the $\epsilon_+$-stable limit, we need the existence of an $(i, j_i, k_i)$ triple where $a_{j_i,k_i} \neq 0$ and $0 \leq j_i \leq m_0-1$. We claim there exists at least one $(i, j_i, k_i)$ triple where $a_{j_i,k_i} \neq 0$ and $0 \leq j_i \leq m_0-1$. Suppose there is no such $(i, j_i, k_i)$ triple. Then, all $a_{j_i,k_i}$ could be zero, in which case all of the sections are of the form $x^{m_0}\alpha_i(x)$ and the generic fibre violates $\epsilon_+$-stability. Thus we have at least one $(i, j_i, k_i)$ triple where $a_{j_i,k_i} \neq 0$. Now, suppose that all triples with $a_{j_ik_i} \neq 0$ have $j_i \ge m_0$. Then, there is a factor of $x^{m_0}$ in each $\psi_i$ section, and the generic fibre again violates $\epsilon_+$-stability. Thus, we have at least one $(i, j_i, k_i)$ triple where $a_{j_i,k_i} \neq 0$ and $0 \leq j_i \leq m_0-1$. To get an $\epsilon_+$-stable limit in the special fibre where $s = 1, t = \pi = 0$, we need the exponent of $t$ for this triple to be zero, getting rid of the dependence on $t$, so the length of the basepoint is bounded above by some $j_i \leq m_0-1$.

In addition, to get valid sections $\psi_i$ of $\cO(-m_0E)$, we need no poles in the $t$ variable to ensure the sections extend over $t = \pi = 0$. To guarantee this, we need that for every $(i, j_i, k_i)$ triple where $a_{j_i,k_i} \neq 0$ and $0 \leq j_i \leq m_0-1$,  that $b \le \frac{k_i+1}{m_0 -j_i}$. 

We choose $b$ as follows. First, find all $(i, j_i, k_i)$ triples where $a_{j_i,k_i} \neq 0$ and $0 \leq j_i \leq m_0-1$; there must be at least one such triple. Then, we have 
$$b : = \min_{\substack{(i, j_i, k_i)\\a_{j_i,k_i} \neq 0\\
0 \leq j_i \leq m_0-1}} \left\{\frac{k_i + 1}{m_0-j_i}\right\} > 0.$$
For this choice of $b \in \bQ_{> 0}$, we get $b \le \frac{k_i+1}{m_0 -j_i}$ for each $(i, j_i, k_i)$ triple where $a_{j_i,k_i} \neq 0$ and $0 \leq j_i \leq m_0-1$. For all other triples where $j_i \ge m_0$,  we need $b \ge \frac{k_i+1}{m_0 -j_i}$, which is trivially satisfied since the ratio $\frac{k_i+1}{m_0 -j_i}$ is negative.

To analyze the filling where $s = t = 0$, our sections $\psi_i$ become $\left(\frac{v}{u}\right)^{m_0} \alpha_i\left(t^b \cdot \frac{v}{u}\right)$ where $\alpha_0(0) = 1$, and we get an $\epsilon_0$-semistable degree $m_0$ rational tail with a length $m_0$ basepoint at $[1:0]$.
\end{proof}

The following lemma ensures that the family of curves used in the construction above is indeed a flat family of nodal curves.
\begin{lem} \label{L:flatfamilyofnodalcurves}
    The family $\mathrm{Bl}_{(x, t^b)}(\bA_x^1 \times \STr)$ is a flat family of nodal curves for $b \in \bQ_{> 0}$.
\end{lem}
\begin{proof}
    We check that each geometric fibre is a nodal curve. This is an open condition so we check the fibre at the origin. In the chart $U_1$, the ring is
    $$R\left[x, s, t, \frac{u}{v}\right] \big/ \left(s, t, t^b - \frac{u}{v}x, st-\pi\right) \cong K\left[x,\frac{u}{v}\right] \big/ \left(x\cdot\frac{u}{v}\right)$$
    and we have a node. In the chart $U_2$, we get
    $$R\left[x, s, t, \frac{v}{u}\right]\big/ \left(s, t, x-t^b\frac{v}{u}, st-\pi\right) \cong K\left[\frac{v}{u}\right]$$
    and the special fibre of the family is thus a nodal curve: $\bA^1$ with a rational tail. Flatness of the family follows since $(s, t)$ is a regular sequence. \end{proof}
\begin{rem}\label{R:+unique}
    We've constructed a $\epsilon_+$-stable extension over the copy of $\Spec(R)$ where $s=1, t= \pi$. This is the unique filling since the $\epsilon_+$-stable locus is separated.
\end{rem}

\subparagraph{Case 2b: $\epsilon_0$-semistable over $(1,0)$ and $\epsilon_-$-stable over $(0,1)$.} We construct an $\epsilon_0$-semistable filling here via a similar blowup, and we show that blowing up at an ideal of the form $(x, t^b)$ is the only way to introduce a rational tail satisying the stability condition.

\begin{prop}\label{T:epsilon-0-}
    The quasimap to $\bP^N$ from the family of nodal curves given by $\mathrm{Bl}_{(x, t^b)}(\bA_x^1 \times \STr)$ for certain $b \in \bQ_{>0}$ is $\epsilon_-$-stable over $(0,1)$, $\epsilon_0$-semistable over $(1,0)$ and is $\epsilon_0$-semistable over $s = t = 0$, after passing to a ramified cover.
\end{prop}
\begin{proof}
We use the same approach as \Cref{T:epsilon-pm}. We start with the same intial $\epsilon_-$-stable data and construct the $N+1$ sections $\psi_i$ of $\cO(-m_0E)$. We again pass to a ramified cover, adjoining $\sqrt[c]{\pi}$ for $c \in \bN_{> 0}$ in order to make sense of $b \in \bQ_{>0}$. As before, we focus on the chart $U_2$ away from the node. We have:
\begin{align*}
    \psi_i\left(t^b \cdot \frac{v}{u}, \pi\right) \cdot f^{-m_0}
    &= \left[\left(\frac{v}{u}\right)^{m_0} \alpha_i\left(t^b \cdot \frac{v}{u}\right) + st^{1-bm_0} \beta_i\left(t^b \cdot \frac{v}{u}, st\right)\right] \cdot u^{m_0}
\end{align*}

To get an $\epsilon_0$-semistable limit where $s = 1,t = \pi$, we need the sections $\psi_i$ to have a basepoint of length $\leq m_0$. Since generically we have $\epsilon_-$ stability, we have that $\alpha_0(0) = 1$ up to a linear change of coordinates. Thus, the common order of vanishing of the $\psi_i$'s with respect to the $\frac{v}{u}$ coordinate in this chart is bounded above by $m_0$. 

To get valid sections $\psi_i$ of $\cO(-m_0E)$, we need no poles in the $t$ variable. If we have an $(i, j_i, k_i)$ triple where $a_{j_i,k_i} \neq 0$ and $0 \le j_i < m_0$, then generically we have $\epsilon_+$ stability and we choose
$$ 0 < b \leq \min_{\substack{(i, j_i, k_i)\\a_{j_i,k_i} \neq 0}} \left\{\frac{k_i + 1}{m_0-j_i} \right\}.$$
Otherwise, if all $(i, j_i, k_i)$ triples where $a_{j_i,k_i} \neq 0$ are such that $j_i \ge m_0$, we are free to choose any $b \in \bQ_{>0}$. As well, if all $a_{j_i,k_i} = 0$, then any choice of $b > 0$ will work.

To analyze the filling where $s = t = 0$, our sections become just $\left(\frac{v}{u}\right)^{m_0} \alpha_i\left(t^b \cdot \frac{v}{u}\right)$ where $\alpha_0(0) = 1$, and we get an $\epsilon_0$-semistable degree $m_0$ rational tail with a length $m_0$ base point at $[1:0]$.
\end{proof}

\begin{lem}\label{L: unqiue ideal}
    Blowing up at an ideal of the form $(x, t^b)$ is the only way to blowup an ideal supported at $x=t=0$ and get a reduced flat family of nodal curves over $\Spec(R[s, t]/(st-\pi))$.
\end{lem}
\begin{proof}
    We show that blowing up at another ideal supported at the point $x = t = 0$ will not result in every fibre over $\Spec(R[s, t]/(st-\pi))$ being reduced. Blowing up at an ideal $I$ supported at $\{x = t = 0\}$ means that $(x, t) \subseteq \sqrt{I}$. Thus, we are looking at ideals $(x^a, t^b) \subseteq I \subsetneq \bC[x, t]$ where $a, b \ge 1$. Working in the local blowup chart $U_2$ where $u \neq 0$, we have that the central fiber corresponds to the ring
    $$K[x, v]/ (x^a - t^bv) \cong K[x, v]/ (x^a) $$
    which is not reduced if $a > 1$. Thus we must have that $a = 1$. Any other ideal $I$ with $(x, t^b) \subsetneq I \subsetneq \bC[x, t]$ is also of the form $(x, t^b)$.
\end{proof}

\paragraph{Case 3: Blowing up at the strict transform of both the $s$-axis and the $t$-axis.} For this case, note that both a family of $\epsilon_+$-stable and $\epsilon_0$-semistable quasimaps have an underlying $\epsilon_-$-stable family via the contraction morphisms in \Cref{T: contraction_pm} and \Cref{T: e_0-contraction}, respectively. We build off of Case 2.

\subparagraph{Case 3a: Strictly $\epsilon_0$-semistable over $(1,0)$ and $\epsilon_+$-stable over $(0,1)$.} Consider the situation where over the special fibre of one copy of $\Spec(R)$, we get the strictly $\epsilon_0$-semistable limit where a degree $m_0$ rational tail with a length $m_0$ basepoint develops. On the other special fibre, we get the $\epsilon_+$-stable limit with a degree $m_0$ rational tail. Over $\Spec(K)$, we have $\epsilon_+$-stability. 

\begin{prop}
\label{T:epsilon-0+}
    The quasimap to $\bP^N$ from the family of nodal curves given by $\mathrm{Bl}_{(x, s^at^b)}(\bA_x^1 \times \STr)$ for certain $a, b \in \bQ_{>0}$ is $\epsilon_0$-semistable over $(1,0)$, $\epsilon_+$ stable over $(0,1)$, and is  $\epsilon_0$-semistable over $s = t = 0$, after passing to a ramified cover.
\end{prop}
\begin{proof}
Via the contraction morphisms in \Cref{T: contraction_pm} and \Cref{T: e_0-contraction}, we start with the data of an $\epsilon_-$-stable family over both copies of $\Spec(R)$. As in previous cases above, this is given by the data $(\bA^1_R, \cO_{\bA^1_R}, \varphi_i(x, \pi) = x^{m_0} \alpha_i(x) + \pi\beta_i(x, \pi))$ where $\alpha_0(0) = 1$ for $0 \le i \le N$.

Let $E \cong \bP^1$ be the exceptional divisor of the blowup at the ideal $(x, s^at^b)$ with coordinates $[u :v]$. The blowup has equation $xu = s^at^bv$. We work in the chart $U_2$ that does not contain the node, where $u \neq 0$, with coordinates $s, t$ and  $\frac{v}{u}$. We adjoin $\sqrt[c]{\pi}$ for $c \in \bN_{> 0}$ in order to make sense of $a, b \in \bQ_{>0}$. The line bundle giving the quasimap after the blowup is $\cO(-m_0E)$. We describe $N+1$ sections $\psi_i$ of $\cO(-m_0E)$. We have:

\begin{align*}
    \psi_i\left(s^at^b \cdot \frac{v}{u}, \pi \right) \cdot f^{-m_0} &= \left[s^{am_0}t^{bm_0}\left(\frac{v}{u}\right)^{m_0} \alpha_i\left(s^at^b \cdot \frac{v}{u}\right) + st \beta_i\left(s^at^b \cdot \frac{v}{u}, st\right)\right] \cdot \left(\frac{s^at^b}{u}\right)^{-m_0}\\
    &= \left[\left(\frac{v}{u}\right)^{m_0} \alpha_i\left(s^at^b \cdot \frac{v}{u}\right) + s^{1-am_0}t^{1-bm_0} \beta_i\left(s^at^b \cdot \frac{v}{u}, st\right)\right] \cdot u^{m_0}
\end{align*}

Without loss of generality, we construct the $\epsilon_+$ stable limit over $(0,1)$ as in \Cref{T:epsilon-pm}, by taking
$$a : = \min_{\substack{(i, j_i, k_i)\\a_{j_i,k_i} \neq 0\\
0 \leq j_i \leq m_0-1}} \left\{\frac{k_i + 1}{m_0-j_i}\right\} > 0.$$
We also construct the $\epsilon_0$-semistable limit over $(1,0)$ as in \Cref{T:epsilon-0-}. Note that we have an $(i, j_i, k_i)$ triple where $a_{j_i,k_i} \neq 0$ and $0 \le j_i < m_0$, due to generic $\epsilon_+$-stability. To construct the $\epsilon_0$-semistable limit, we take $0 < b \le a$.

Lastly, we analyze the filling where $s=t=0$. Our sections $\psi_i$ become $\left(\frac{v}{u}\right)^{m_0} \alpha_i\left(t^b \cdot \frac{v}{u}\right)$ where $\alpha_0(0) = 1$, and we get an $\epsilon_0$-semistable degree $m_0$ rational tail with a length $m_0$ basepoint at $[1:0]$.

\end{proof}
\begin{lem}
    The family $\mathrm{Bl}_{(x, s^at^b)}(\bA_x^1 \times \STr)$ is a flat family of nodal curves for $a, b \in \bQ_{\ge 0}$.
\end{lem}
\begin{proof}
    This proof is very similar to \Cref{L:flatfamilyofnodalcurves}.
\end{proof}

\begin{rem}
    We've constructed a $\epsilon_+$-stable family over one copy of $\Spec(R)$. This is the unique filling, as in \Cref{R:+unique}. The other side with the $\epsilon_0$-semistable filling is unique since we must blowup at an ideal of the form $(x, s^at^b)$, as in \Cref{L: unqiue ideal}.
\end{rem}

\subparagraph{Case 3b: Strictly $\epsilon_0$-semistable on $(0,1)$ and $(1,0)$.} For our final case, we have
\begin{prop}
    The quasimap to $\bP^N$ from the family of nodal curves given by $\mathrm{Bl}_{(x, s^at^b)}(\bA_x^1 \times \STr)$ for certain $a, b \in \bQ_{>0}$ is $\epsilon_0$-semistable over $(1,0)$ and $(0,1)$ and is $\epsilon_0$-semistable over $s = t = 0$, after passing to a ramified cover.
\end{prop}
\begin{proof}
This is very similar to the proof of \Cref{T:epsilon-0+}. If we have an $(i, j_i, k_i)$ triple where $a_{j_i,k_i} \neq 0$ and $0 \le j_i < m_0$, then generically we have $\epsilon_+$-stability and we choose
$$ 0 < a, b \leq \min_{\substack{(i, j_i, k_i)\\a_{j_i,k_i} \neq 0}} \left\{\frac{k_i + 1}{m_0-j_i} \right\}.$$
If all $(i, j_i, k_i)$ triples where $a_{j_i,k_i} \neq 0$ are such that $j_i \ge m_0$, we a free to choose any $a, b \in \bQ_{>0}$. As well, if all $a_{j_i,k_i} = 0$, then any choice of $a, b > 0$ will work.
\end{proof}

\paragraph{Summary of S-Completeness Casework.} \label{sec:S-complete-summary} We summarize the local fillings above for $\bigQmap$. In particular, we synthesize how particular rational numbers $(a,b)$ in the ideals $(x, s^at^b)$ used in the blowups above control the geometry of the quasimap. After contracting rational tails of degree $m_0$ via \Cref{T: e_0-contraction}, we have an underlying $\epsilon_-$-stable family over $\Spec(R)$, with sections 
$$\varphi_i(x, \pi) = x^{m_0} \alpha_i(x) + \pi\beta_i(x, \pi) = \sum_{r_i}a_{r_i}x^{m_0+r_i} + \sum_{j_i, k_i}a_{j_i, k_i}x^{j_i}\pi^{k_i+1}$$
for $0 \le i \le N$. If the generic fibre satisfies $\epsilon_+$-stability, we have the following invariant

$$\lambda_\mathcal{L} := \min_{\substack{(i, j_i, k_i)\\a_{j_i,k_i} \neq 0 \\ 0 \leq j_i \leq m_0-1}} \left\{\frac{k_i + 1}{m_0-j_i} \right\}.$$

Alternatively, we can extract the rational number $\lambda_\cL$ from a valuation on the underlying $\epsilon_-$-stable quasimap sections. Consider the valuation $\mathcal{V}_t$ defined by $\mathcal{V}_t(\pi) = 1$ and $\mathcal{V}_t(x) = t$. Then
$$\lambda_\cL = \max_t \left\{\mathcal{V}_t(\varphi_i(x, \pi)) = m_0t \text{ for all } i\right\}.$$

To detect which blowup was performed, we use a similar valuation on the sections $\psi_i(\pi, v)$ associated to the far point on the new rational tail. Consider the valuation $\mathcal{V}_s$ where $\mathcal{V}_s(\pi) = 1$ and $\mathcal{V}_s(x) = s$. Then
$$\lambda'_\cL = \max_s \left\{\mathcal{V}_s(\psi_i(\pi, v)) = m_0s \text{ for all } i\right\}.$$
To recover the number $b$ corresponding to the blowup at the ideal $(x, \pi^b)$, take the difference $ \lambda_\cL - \lambda'_\cL$.

All possible configurations of the two special fibres in the local $\bA^1_R$ model, together with the ideal corresponding to the blow‑up and the constraints on the parameters $a,b$ identifying the filling, are listed in the following table.
\renewcommand{\arraystretch}{1.25}
\begin{table}[H]
\centering
\begin{tabular}{|l|l|l|l|p{3.50cm}|}
\hline Case & Special Fibres & Generic Fibre & Filling & Numerics \\
\hline (2a) & $\epsilon_-$ / $\epsilon_+$ & $\epsilon_-$ and $\epsilon_+$ & $(x, t^b)$ or $(x, s^b)$& $b = \lambda_\cL$\\ \hline
(2b) & $\epsilon_-$ / $\epsilon_0$ & $\epsilon_-$ and $\epsilon_+$ & $(x, t^b)$ or $(x, s^b)$ & $0 < b < \lambda_\cL$\\ \hline
(2b) & $\epsilon_-$ / $\epsilon_0$ & $\epsilon_-$ & $(x, t^b)$ or $(x, s^b)$& $b > 0$\\ \hline
(3a) & $\epsilon_+$ / $\epsilon_0$ & $\epsilon_-$ and $\epsilon_+$ & $(x, s^a t^b)$ & $0 < b < a = \lambda_\cL$ or $0 < a < b = \lambda_\cL$\\ \hline
(3b) & $\epsilon_0$ / $\epsilon_0$ & $\epsilon_-$ and $\epsilon_+$ & $(x, s^at^b)$ & $0 < a, b < \lambda_\cL$\\\hline
(3b) & $\epsilon_0$ / $\epsilon_0$ & $\epsilon_-$ & $(x, s^at^b)$ & $a, b > 0$\\ \hline
\end{tabular}
\end{table}

\subsection{$\Theta$-Reductivity}

As done above for $S$-completeness, we establish $\Theta$-reductivity for $\bigQmap$ by analyzing a local filling problem in formal neighbourhoods of a finite number of points on the source curve.

\begin{thm} \label{T:Theta-red} The algebraic stack $\bigQmap$ is $\Theta$-reductive.
\end{thm}

We reduce to the case of filling over $\Spec(R[t])$ by \Cref{R:fill-without-Gm-action}, and uniqueness of the filling follows from \Cref{R:unique-filling}. We analyze each component of the source curve separately as discussed in \Cref{L:generically-smooth}. The case where the generic fibre is a single spinning rational tail of degree $m_0$ with a length $m_0$ basepoint (Case I) is is \Cref{P:degenerate-case}. We are left to produce a filling for the case where the generic fibre is $\epsilon_-$-stable (Case II). To show $\Theta$-reductivity holds in this situation, we analyze the local $\bA^1$-filling problem discussed in the above subsection for all possible configurations of $\epsilon_0$-semistable quasimaps over $\Spec(R[t])\setminus 0$. Throughout the analysis, we freely adjoin roots of both $\pi$ and $t$ which is allowed by \Cref{R:adjoin-roots}. This case analysis is very similar to the local $\bA^1$-filling problem discussed above for S-completeness. We outline it here.

\paragraph{Case Analysis for Local Fillings.}
A map from $\Spec(R[t])\backslash 0$ to $\bigQmap$ corresponds to a family of $\epsilon_0$-semistable quasimaps over $\Spec(R)$ and a family of $\epsilon_0$-semistable quasimaps over $\Spec(K[t])$ that agree upon restriction to $\Spec(K)$. We refer to the two special fibres of these families as the $\pi$-special fibre and the $t$-special fibre, respectively.

We start with the underlying $\epsilon_-$ stable family, with sections 
$$\varphi_i(x, \pi, t) = x^{m_0} \alpha_i(x) + \pi t \beta_i(x, \pi, t) = \sum_{r_i}a_{r_i}x^{m_0+r_i} + \sum_{j_i, k_i, \ell_i}a_{j_i, k_i}x^{j_i}\pi^{k_i+1}t^{\ell_i + 1}$$
for $0 \le i \le N$ that are polynomials in $x, \pi$ and $t$, developing a length $m_0$ basepoint at $\pi=t =x=0$. 
The filling is constructed via a blowup in a very similar way as done for $S$-completeness. We blow up $\mathbb{A}^1_x\times\Theta_R $ at a strict transform of either the $\pi$-axis or the $t$ axis, or both, to construct the desired filling over $\pi = t = 0$. Depending on whether the $\pi$-special fibre and/or the $t$-special fibre requires a rational tail, we blow up at an ideal involving $\pi$, $t$, or both to produce a flat family of nodal curves. As in \Cref{L: unqiue ideal}, the only ideals yielding a reduced flat family of nodal curves are those of the form $(x,\pi^a t^b)$.

\begin{lem}
    The family $\mathrm{Bl}_{(x, \pi^a t^b)}(\bA_x^1 \times \Theta_R)$ is a flat family of nodal curves for $a, b \in \bQ_{\ge 0}$.
\end{lem}
\begin{proof}

This is similar to \Cref{L:flatfamilyofnodalcurves}. Flatness follows since $(\pi, t)$ is a regular sequence. \end{proof}

Similar to $S$-completeness, the fillings of the quasimap sections are controlled by two different invariants in the case where the generic fibre satisfies $\epsilon_+$-stability, one for each axis. If the generic fibre satisfies $\epsilon_+$-stability, define 
$$\lambda^{\pi}_\mathcal{L} := \min_{\substack{(i, j_i, k_i)\\a_{j_i,k_i, \ell_i} \neq 0 \\ 0 \leq j_i \leq m_0-1}} \left\{\frac{k_i + 1}{m_0-j_i} \right\} \hspace{0.5cm} \mathrm{and}\hspace{0.5cm} \lambda^{t}_\mathcal{L} := \min_{\substack{(i, j_i, \ell_i)\\a_{j_i,k_i, \ell_i} \neq 0 \\ 0 \leq j_i \leq m_0-1}} \left\{\frac{\ell_i + 1}{m_0-j_i} \right\}.$$

\paragraph{Summary of $\Theta$-Reductivity Casework.} The following chart summarizes all of the casework for the local model in the analysis of $\Theta$-reductivity of $\bigQmap$. 
\renewcommand{\arraystretch}{1.25}
\begin{table}[H]
\centering
\begin{tabular}{|l|l|l|l|l|p{2.25cm}|}
\hline Case & $t$-Special Fibre & $\pi$-Special Fibre & Generic Fibre & Filling & Numerics \\  \hline
(2a) & $\epsilon_+$ & $\epsilon_-$ & $\epsilon_-$ and $\epsilon_+$ & $(x, t^b)$ & $b = \lambda^t_\cL$\\ \hline
(2a) & $\epsilon_-$ & $\epsilon_+$ & $\epsilon_-$ and $\epsilon_+$ & $(x, \pi^a)$ & $a = \lambda^\pi_\cL$\\ \hline
(2b) & $\epsilon_0$  & $\epsilon_-$ & $\epsilon_-$ and $\epsilon_+$ & $(x, t^b)$ & $0 < b < \lambda^t_\cL$\\ \hline
(2b) & $\epsilon_-$  & $\epsilon_0$ & $\epsilon_-$ and $\epsilon_+$ & $(x, \pi^a)$ & $0 < a < \lambda^\pi_\cL$\\ \hline
(2b) & $\epsilon_0$  & $\epsilon_-$ & $\epsilon_-$ & $(x, t^b)$ & $b>0 $\\ \hline
(2b) & $\epsilon_-$  & $\epsilon_0$ & $\epsilon_-$ & $(x, \pi^a)$ & $a>0 $\\ \hline
(3a) & $\epsilon_+$  & $\epsilon_0$ & $\epsilon_-$ and $\epsilon_+$ & $(x, \pi^a t^b)$ & $a = \lambda^\pi_\cL$ \newline $0 < b < \lambda^t_\cL$\\ \hline
(3a) & $\epsilon_0$  & $\epsilon_+$ & $\epsilon_-$ and $\epsilon_+$ & $(x, \pi^a t^b)$ & $0 < a < \lambda^\pi_\cL$ \newline $b = \lambda^t_\cL$\\ \hline
(3b) & $\epsilon_0$  & $\epsilon_0$ & $\epsilon_-$ and $\epsilon_+$ & $(x, \pi^at^b)$ & $0 < a < \lambda^\pi_\cL$ \newline $0 < b < \lambda^t_\cL$\\ \hline
(3b) & $\epsilon_0$  & $\epsilon_0$ & $\epsilon_-$ & $(x, \pi^at^b)$ & $ a >0$ \newline $b > 0$\\ \hline
\end{tabular}
\end{table}

\subsection{Properness of the Good Moduli Space}

In this section, we prove the properness of the good moduli space for $\bigQmap$. To do this, it suffices to show the existence part of the valuative criterion for properness holds by \cite[Theorem A]{Alper_2023_Existence}. 

\begin{thm}\label{T:GMS-proper}
    The good moduli space for $\bigQmap$ is proper over $\mathbb{C}$.
\end{thm}
\begin{proof}
First, we reduce to the case where the generic fibre is smooth, and work component by component in the total space of the family, as done in \Cref{L:generically-smooth} for the proof of $\Theta$-reductivity and $S$-completeness.  To do this, we may need to extend $R$ so we can assume all nodes are $K$-points. To find a filling, we analyze the extension problem separately for each component to obtain a filling for the whole family. There as two main cases to consider.

First consider the case, when $(g, n, d) \neq (0, 1, m_0)$ and the generic fibre is not a degree $m_0$ rational tail. Then the generic fibre is $\epsilon_-$-stable. By properness of $\Qmapminus$, we have a unique extension to a family of $\epsilon_-$-stable quasimaps over $\Spec(R)$, up to finite base change.

The second case to consider is when the generic fibre is a single rational tail of degree $m_0$. Since rational tails of degree $< m_0$ are not allowed, the curve cannot further degenerate into irreducible components. If there is no length $m_0$ basepoint, the quasimap is $\epsilon_+$-stable and the extension over $\Spec(R)$ exists by properness of $\Qmapplus$, up to finite base change. If there is a length $m_0$ basepoint, we require a separate argument, which will be covered in \Cref{L:single-tail-proper} below.
\end{proof}

\begin{lem}\label{L:single-tail-proper}
     Let $(C_K, x_K, L_K, \{s_i\}_{i=0}^N)$ be an $\epsilon_0$-semistable quasimap over $\Spec(K)$ consisting of a single rational tail of degree $m_0$ with a marked point $x_K$ and a basepoint $y_K \neq x_K$ of length  $m_0$. Then, after a finite base change $R \subset R'$ where $R' = R[\pi^{1/m_0}]$, there exists an extension of this quasimap to an $\epsilon_0$-semistable quasimap over $\Spec(R')$.
\end{lem}
\begin{proof}

We will construct the extension explicitly. Chose an isomorphism $C_K \cong \bP^1$ such that $x_K \in C_K(K)$ denotes a marked point at $[0:1]$, and $y_K \in C_K(K)$ denotes a basepoint of length $m_0$ occuring at $[1:0]$. Under this isomorphism, $L_K \cong \mathcal{O}_{\bP^1}(m_0)$. Since $y_K$ is a basepoint of length $m_0$, all sections vanish to order $m_0$ at $[1:0]$. In homogeneous coordinates $[x:y]$, a section of $\mathcal{O}(m_0)$ vanishing to order $m_0$ at $[1:0]$ must be divisible by $y^{m_0}$. Hence we can write
\begin{equation*}
\{s_i\}_{i = 0}^N = \{\alpha_i y^{m_0}\}_{i = 0}^N
\end{equation*}
with $\alpha_i \in K$, and at least one $\alpha_i \neq 0$.

First, we extend the family of curves over $\Spec(R)$. Consider the family $(C_K, x_k, y_K)$. A family of $\mathbb{P}^1$'s over $\Spec(K)$ with two special points corresponds to a $K$-point of $\mathcal{M}_{0,2} \cong B\bC^*$. Since $B\bC^*$ is proper, this $K$-point extends uniquely to an $R$-point. Concretely, we obtain a family of $\bP^1$'s with two sections corresponding to $x_K$ and $y_K$. The extension is $C \cong \mathbb{P}_R(\mathcal{N} \oplus \mathcal{O}_R)$ for some line bundle $\mathcal{N}$ on $\Spec(R)$ determined by the original family. 
The line bundle extends as $L:= \iota_*L_K$ where $\iota: \cC_K \hookrightarrow C$ denotes the inclusion, since the restriction map $\Pic(\bP_{R}(\mathcal{N} \oplus \mathcal{O}_{R})) \xrightarrow{} \Pic(C_K)$ is an isomorphism.

We now have a family $(C, L)$ over $\Spec(R)$. The sections of the quasimap on the generic fiber are $\{s_i = \alpha_i y^{m_0}\}_{i=0}^N$ with $\alpha_i \in K$. To extend them to sections over $\Spec(R)$, we would need $\alpha_i \in R$ for all $i$. If this is not already the case, we perform a finite base change. Let $R' = R[\pi^{1/m_0}]$ with fraction field $K'$. Consider a valuation $\nu$ on $K'$, normalized by $\nu(\pi') = 1$. Set $\lambda \in K'^\times$ with
$$\nu(\lambda) := -\min_i\nu(\alpha_i)$$
which yields $\nu(\lambda\alpha_i)\ge 0$ for all $i$, so $\lambda\alpha_i \in R'$. Define $\tilde{s_i}$ as the extension of the sections $s_i$ as follows
$$\{\tilde{s_i}\}_{i = 0}^N = \{\lambda\alpha_i y^{m_0}\}_{i = 0}^N$$
where $\lambda\alpha_i \in R'$ for all $i$. Since at least one coefficient $\lambda\alpha_i$ is non-zero, the sections are not all zero. At the point $y_0 = [1:0]$ , each $\tilde{s}_i$ vanishes to order $m_0$. Hence $\ell(y_0) = m_0$, and there are no other basepoints.
\end{proof}

\section{Semistable Quasimaps in Affine Degeneration Space}\label{S:ADeg}

In this section, we connect our construction of $\epsilon_0$-semistable quasimap fillings to the affine degeneration space, introduced by Halpern-Leistner in \cite[Section 4]{ICM2026}.

\subsection{Affine Degeneration Space}
The affine degeneration space is a metric space  associated to a $K$-point of an algebraic stack. Up to adjoining roots of the uniformizer, the affine degeneration space $\ADeg(\frak{X}, \eta)$ of a $K$-point $\eta$ of an algebraic stack $\frak{X}$ consists of all the ways to extend from $K$ over $R$.

\begin{const}[Construction 4.1, \cite{ICM2026}]
    Let $R$ be a complete DVR with unformizer $\pi$ and fraction field $K$. Let $\mathfrak{X}$ be an algebraic stack, and $\eta \in \mathfrak{X}(K)$. Consider the model for the standard $n$-simplex $\Delta^n \subset \mathbb{R}^{n+1}$:
$$\STr^n := \Spec\!\Big(R[t_0^r,\dots, t_n^r, \forall r \in \mathbb{Q}_{\geq 0}]/(t_0\cdots t_n - \pi)\Big) \big/ \mathbb{G}_n $$
where $\mathbb{G}_n = \Spec\!\big(\mathbb{Z}[z^r, \forall r \in W_n]\big)$. A morphism $\STr^0 \to \frak{X}$ corresponds to a map \\$\Spec\!\big(R[\pi^{1/m}]\big) \to \frak{X}$ for some $m$. In addition, any rational linear map $\phi : \Delta^m \to \Delta^n$, defined by a matrix with rational coefficients $\phi_{i,j} \geq 0$ such that
$\phi_{i,0} + \cdots + \phi_{i,n} = 1$, for all $i = 0,\dots, m$, induces a canonical map $\STr^m \to \STr^n$ given in coordinates by $t_j^r \;\longmapsto\; t_0^{r\phi_{0,j}} \cdots t_m^{r\phi_{m,j}}$.
A morphism $\STr^n \to X$ can be thought of as giving a family of maps from (totally ramified covers of) $\Spec R$ parameterized by rational points of $\Delta^n$.

The open substack where $t_0\cdots t_n \neq 0$ is isomorphic to $\Spec K'$, where $K' = K[\pi^{1/m}, \forall m > 0]$, and we call this generic point $\eta$. Given a stack $\frak{X}$ with a point $\xi \in \frak{X}(K)$, we define
$$
\mathbb{S}(X,\xi)_n \;:=\; \left\{\;
\sigma : \STr^n \to \frak{X} \text{ with isomorphism } \sigma(\eta) \simeq \xi|_{K'}
\;\right\}.$$
The affine degeneration space, $\ADeg(\frak{X} , \xi)$ is the topological space obtained by gluing together copies of the standard simplices $\Delta^n$ along rational linear maps, with one copy of $\Delta^n$ for each $\sigma \in \mathbb{S}(X,\xi)_n$.
\end{const}
The following pieces of the dictionary described in \cite[Section 4]{ICM2026} between properties of $\frak{X}$ and properties of $\ADeg(\frak{X} , \xi)$ will be useful to us.

\begin{enumerate}[itemsep = 2pt]
    \item $S$-completeness of $\frak{X}$ means that any two points of $\ADeg(X,\xi)$ are connected by a unique geodesic line segment.
    \item A morphism $\gamma : \Theta_R \to \frak{X}$ with an isomorphism $\gamma(1_K) \simeq \xi$ corresponds to a geodesic ray in $\ADeg(\frak{X},\xi)$.
    \item If $\frak{X}$ is $S$-complete, $\Theta$-reductive, and quasi-compact with affine diagonal, $\ADeg(\frak{X},\xi)$ is expected to to be a a complete $\mathrm{CAT}(0)$ space.
\end{enumerate}

\subsection{Application to $\epsilon_0$-Semistable Quasimaps}

We interpret our $S$-completeness and $\Theta$-reductivity constructions above in terms of the affine degeneration space $\operatorname{ADeg}(\bigQmap,\xi)$. Let $\xi\in \bigQmap(K)$ be a $K$-point of
$\bigQmap$ which is $\epsilon_-$-stable, consisting of the data $(\cC, \{x_i\}_{i = 1}^n,\cL, \{s_j\}_{i = 0}^N)$.

\paragraph{Local Description.} As shown in \Cref{sec:S-complete}, each potential node where a degree $m_0$ rational tail may form in the special fibre corresponds to a family of fillings, parameterized by $b \in [0,\lambda_\cL]$ where $\lambda_\cL \in \bQ$ is extracted from a valuation on the underlying $\epsilon_-$-stable sections.

Geometrically, the $\epsilon_0$-semistable extension over the special point of $\Spec(R)$ of this family is obtained by blowing up along the ideal $(x,\pi^b)$, where $x$ is the local coordinate at the node. The parameter $b$ interpolates between the two endpoints. Suppose the generic fibre also happens to be $\epsilon_+$-stable. Then, we have the following:
\begin{itemize}
    \item $b = 0$: no blow‑up; the limit is $\epsilon_-$-stable and we get a basepoint of length $m_0$.
    \item $b = \lambda_{\mathcal{L}}$: the limit is $\epsilon_+$-stable and we get a rational tail of degree $m_0$ with no length $m_0$ basepoint.
    \item $0 < b < \lambda_{\mathcal{L}}$: the limit is strictly $\epsilon_0$-semistable and we get a degree $m_0$ rational tail containing a length $m_0$ basepoint.
\end{itemize}

In the affine degeneration space, this family corresponds to a interval of length $\lambda_{\mathcal{L}}$ connecting the two endpoints, representing the unique $\epsilon_-$ and $\epsilon_+$-stable limits. 

If the generic fibre is not $\epsilon_+$-stable, then  the blow‑up parameter $b$ may be chosen arbitrarily in $\bQ_{>0}$. In the affine degeneration space, this corresponds to a geodesic ray $[0,\infty)$ emanating from the $\epsilon_-$-stable endpoint, with no upper bound on the parameter $b$. This family comes from a map $\Theta_R \to \bigQmap$.

\paragraph{Global Description.} Translating this local description to the global picture for the full nodal curve that is generically $\epsilon_-$-stable, we get that each point where a degree‑$m_0$ rational tail may form contributes an independent interval $[0,\lambda_{\cL}^{(i)}]$ in the affine degeneration space.  The blow‑up constructions at distinct nodes are independent, which gives
$$\prod_{i \in I} [0,\lambda_{\mathcal{L}}^{(i)}] \;\times\; \prod_{j \in J} [0,\infty)$$
where $I$ indexes nodes lying on a component that is generically $\epsilon_+$-stable, and $J$ indexes nodes where the generic fibre of the component is not $\epsilon_+$-stable.
Thus, locally near $\xi$, the affine degeneration space is a $\mathrm{CAT}(0)$ space, as a product of intervals and rays. This matches the expected structure of the affine degeneration space for a stack that is S-complete and $\Theta$-reductive with affine diagonal, giving an example of the construction introduced in \cite[Section 4]{ICM2026}.

\section{$\Theta$-Stratifications of $\epsilon_0$-Semistable Quasimaps} \label{S:theta-stratifications}
We first recall the definition of a $\Theta$-stratification below, and the rest of the theory of $\Theta$-stability and numerical invariants introduced in \cite{halpernleistner2022structureinstabilitymodulitheory} will be presented throughout the rest of the section as needed. We first recall the definition of a filtration of a point of an algebraic stack $\frak{X}$ with quasi-affine diagonal and locally of finite presentation over a scheme $T$.

\begin{notn}
Let $k$ be a field. Denote by $0 \in \bA^1_k$ the fixed point given by the vanishing of the coordinate $t$, and by $0$ it's image in $\Theta_k:= \bA^1_k/\bG_m$. Similarly, let $1$ be the $k$-point of $\bA^1_k$ given by the vanishing of $t-1$, and denote its image in the quotient $\Theta_k$ also by 1.
\end{notn}

\begin{defn} Let $p : \Spec(k) \rightarrow \frak{X}$ be a $k$-point of $\frak{X}$. A
\textbf{filtration} of $p$ consists of a morphism $f:\Theta_k \xrightarrow{} \frak{X}$ together with an
isomorphism $f(1) \simeq p$. The filtration is called \textbf{non-degenerate} if the induced $\bG_m$-action on the central fiber $f(0)$ has finite kernel.
\end{defn}

\begin{ex}
    A example of a degenerate filtration for a point of $\bigQmap$ would be a filtration of an $\epsilon_0$-semistable quasimap that is both $\epsilon_-$ and $\epsilon_+$-stable. The only filtrations of such objects are trivial. 
\end{ex}

\begin{rem}
    We will only consider non-degenerate filtrations, as these are the ones relevant for $\Theta$-stability. For the rest of this section, we may sometimes omit the adjective ``non-degenerate".
\end{rem}

The stack of filtrations $\Filt(\frak{X}) := \Map(\Theta, \frak{X})$
is represented by an algebraic stack locally of finite presentation over $T$ \cite[Proposition 1.1.2]{halpernleistner2022structureinstabilitymodulitheory}. There is a morphism $\ev_1: \Filt(\frak{X}) \rightarrow \frak{X}$ given by evaluating at $1 \in \Theta$.

\begin{defn}\cite[Definition 2.1.1]{halpernleistner2022structureinstabilitymodulitheory} A \textbf{$\Theta$-stratum} is an open and closed substack $\mathcal{S} \subset \Filt(\frak{X})$ such that $\ev_1: \mathcal{S} \rightarrow \frak{X}$ is a closed immersion.
\end{defn}

\begin{defn}\cite[Definition 2.1.2]{halpernleistner2022structureinstabilitymodulitheory}
    A \textbf{$\Theta$-stratification} of $\frak{X}$ consists of a collection of open substacks $\frak{X}_{\le c}$ indexed by a totally ordered set $\Gamma$ such that:

    \begin{enumerate}
        \item $\frak{X}_{\le c} \subset \frak{X}_{\le c'}$ for $c < c'$;
        \item $\frak{X} = \bigcup_{c\in\Gamma} \frak{X}_{\le c}$;
        \item For all $c$, there exists a $\Theta$-stratum $\mathcal{S}_c \subset \Filt(\frak{X}_{\le c})$ of $\frak{X}_{\le c}$ such that $$\frak{X}_{\le c} \setminus \ev_1(\mathcal{S}_c) = \bigcup_{c' < c} \frak{X}_{\le c'};$$
        \item For every point $p \in \frak{X}$, the set $\{c \in \Gamma \mid p \in \frak{X}_{\le c}\}$ has a minimal element.
    \end{enumerate}
\end{defn}

In this section, we will describe two different $\Theta$-stratifications of $\bigQmap$, allowing us to compare the geometries of $\Qmapminus$ and $\Qmapplus$.

\subsection{Line Bundles on $\bigQmap$}

The central fibre of a non-degenerate filtration, together with its torus action, encodes how a point of $\bigQmap$ can degenerate. By examining the weight of a particular line bundle on this central fibre, we obtain a numerical measure of instability for points of $\bigQmap$. To make this more precise, we start by introducing a family of natural line bundles on $\bigQmap$.

\begin{defn}\label{D: line-bundle} We define the line bundle $\mathcal{M}_{a, b} \in \Pic(\bigQmap)_{\bQ}$ for $a, b \in \bZ_{>0}$ by 
    $$\mathcal{M}_{a, b} := \det(R\pi_*(\cL_{\univ}^a \otimes \omega_\pi^b))^\vee $$
    where $\cL_{\univ}$ is the universal line bundle on the universal curve $\pi: \cC_{\univ} \xrightarrow{} \bigQmap$ and $\omega_\pi$ is the relative dualizing sheaf.
\end{defn}

Given a filtration $f : \Theta_k \xrightarrow{} \bigQmap$ of a $k$-point of $\bigQmap$, the pullback $f^*\mathcal{M}_{a, b}$ acquires a torus action on its central fibre. Denote the weight of this action by $\wt_0(f^*\mathcal{M}_{a, b}) \in \bZ$. We compute this weight, first in the case where the central fibre of the filtration has a single  degree $m_0$ rational tail containing a length $m_0$ basepoint. We will refer to such rational tails as \textbf{spinning}, since they are precisely the components of an $\epsilon_0$-semistable curve that rotate under a non-trivial $\bC^*$ action.

\paragraph{Weight Calculation for a Single Spinning Rational Tail.} Suppose the central fibre consists of the curve $C = C_1 \cup T$ where $T \cong \bP^1$ is a single spinning rational tail attached to the main component of the curve, $C_1$ at the node $p = [0:1]$. We give $\bC^*$ coordinate $q$ and let it act on $T$ by
$$q \cdot [x:y] \mapsto [q^rx:y]$$
with spin rate $r \in \bZ$. We write $\cO_T(\alpha, \beta)$ for the linearized line bundle with weight $\alpha$ at the node $p = [0:1]$ and weight $\beta$ at the other fixed point $[1:0]$. Line bundles on non-spinning components of the curve will have weight zero, so we may take $\alpha = 0$ for bundles restricted from $C$. The degree of the bundle is then $(\alpha-\beta)/r$. Since the degree of $\cL_{\univ}$ restricted to $T$ is $m_0$, we have that $\cL_{\univ}^a|_T$ is $\cO_T(0, -ram_0)$ in linearized notation. Similarily, the dualizing sheaf $\omega_{\pi}^b|_T$ is $\cO_T(0, rb)$. Lastly, we must isolate the weight contribution from the rational tail via tensoring with $\cO_T(-p)$, coming from the short exact sequence $0 \xrightarrow{} \cO_T(-p) \xrightarrow{} \cO_C \xrightarrow{} \cO_{C_1} \xrightarrow{} 0$. The representation $R\Gamma(\cL_{\univ}^a \otimes \omega_\pi^b)$ decomposes into a trivial contribution from $C_1$ and the contribution from $T$, twisted by $\cO_T(-p)$. Thus, the relevant linearized bundle on $T$ is
$$ \cO_T(-1, r(b-am_0))$$
with degree $d' := am_0-b -1$. Computing the weight of $\det(R\Gamma(\cO_T(-1, r(b-am_0)))^\vee$ gives
$$-\dfrac{r}{2}(d'+1)(-1 + b - m_0 a).$$

\begin{rem} \label{R: weight-vanishes} Notably, the weight vanishes when $b \in\{am_0, am_0+1\}$. When analyzing stability, we therefore only consider the line bundles $\mathcal{M}_{a, b}$ of \Cref{D: line-bundle} where $b \notin \{am_0, am_0+1\}$. If $b \in \{am_0, am_0+1\}$, all points of $\bigQmap$ will be semistable.
\end{rem}

\paragraph{Weight Calculation for Multiple Spinning Tails.} We now generalize the weight calculation to a central fibre consisting of a main component $C_1$ together with $s$ spinning rational tails $T_1, \ldots T_s$, attached to $C_1$ at distinct points $p_1, \ldots p_s$. Each tail spins independently with rates $(r_1, \ldots, r_s) \in \bZ^s$. Since the $\bC^*$ action on different spinning tails is independent, the total weight of $\det(R\Gamma(\cL_{\univ}^a \otimes \omega_{\pi}^b))^\vee$ is the sum of the contributions from each tail, which gives
\begin{equation}\label{E:LB-weight}
    \dfrac{-1}{2}\bigl(\sum_{i=1}^s r_i\bigr)(d'+1)(-1 + b - m_0 a).
\end{equation} 

\subsection{Numerical Invariants}

To study $\Theta$-stability for points of $\bigQmap$ and define $\Theta$-stratifications, we use a numerical invariant, a structure on algebraic stacks introduced by Halpern-Leistner. We summarize the definition; a more precise formulation can found in \cite[Definition 0.0.3]{halpernleistner2022structureinstabilitymodulitheory}. 

\begin{defn} A \textbf{numerical invariant} on an algebraic stack $\frak{X}$ is a scale-invariant function that assigns a real number $\nu(f)$ to any non-degenerate filtration $f : \Theta \xrightarrow{} \frak{X}$ as follows: given the data of 
\begin{itemize}
    \item a cohomology class $\ell \in H^2(\frak{X}; \bQ)$
    \item a positive-definite class $b \in H^4(\frak{X}; \bQ)$
\end{itemize}
the numerical invariant is $\nu(f) := f^*\ell/\sqrt{f^*b}\in \bR$. This assignment must be invariant under field extensions, locally constant in families, and compatible with restriction of tori, analogous to conditions (1), (2) and (3) in \Cref{D:norm} below.

\end{defn}

To obtain a scale-invariant $\nu$ as above for stability analysis, we must divide by a suitable norm on graded points of $\bigQmap$. Note that although the following definition involves data for all $q \ge 1$, only the $q = 1$ data will be used to define semistability.

\begin{defn} \cite[Definition 4.1.12]{halpernleistner2022structureinstabilitymodulitheory}\label{D:norm}
A \textbf{norm} $b$ on graded points of $\frak{X}$ is an assignment defined as follows. Let $k$ be a field, $p \in \frak{X}$, and $\gamma:(\mathbb{G}_m^q)_k \to \Aut(p)$ a homomorphism with finite kernel. Then $b$ assigns to this data a positive definite quadratic norm $b_\gamma : \mathbb{R}^q \to \mathbb{R}$ such that:
\begin{enumerate}
\item[(1)] $b_\gamma$ is unchanged under field extensions $k \subset k'$.
\item[(2)] $b$ is locally constant in algebraic families. That is, choose any scheme $T$, a morphism $\zeta: T \xrightarrow{} \frak{X}$ and a homomorphism $\gamma: (\bG_m^q)_T \xrightarrow{} \Aut(\zeta)$ of $T$-group schemes with finite kernel. As we vary $t \in T$, we require that the function $b_{\gamma_t}$ is locally constant on $T$.
\item[(3)] Given a homomorphism $\phi : (\mathbb{G}_m^w)_k \to (\mathbb{G}_m^q)_k$ with finite kernel, the norm $b_{\gamma \circ \phi}$ is the restriction of $b_\gamma$ along the inclusion $\mathbb{R}^w \hookrightarrow \mathbb{R}^q$ induced by $\phi$.
\end{enumerate}
\end{defn}

Consider a filtration of a point of $\bigQmap$ whose central fibre has $s$ spinning tails. The automorphism group includes permutations of these tails, so the norm must invariant under this action of the symmetric group. With this in mind, we now define the norm we will use for stability analysis.

\begin{defn}\label{D:norm =r^2}
Given a filtration $f$ of a point of $\bigQmap$ where the central fibre has $s$ spinning rational tails, we consider the following function
$$f^*b:= ||\vec{r}||^2  = \sum_{i=1}^s r_i^2$$
where $\vec{r} = (r_1, \ldots, r_s)$ records the spin rate of the $s$ spinning tails in the central fibre of the filtration.
\end{defn}

\begin{lem}
\Cref{D:norm =r^2} defines a norm on graded points of $\bigQmap$.
\end{lem}
\begin{proof}
By \cite[Example 4.1.13]{halpernleistner2022structureinstabilitymodulitheory}, it suffices to show that this positive-definite quadratic form arises from a cohomology class in $H^4(\bigQmap, \bQ)$. We can recover $\sum_{i=1}^s r_i^2$ from the following cohomology class:
$$2\ch_2(R\pi_*(\cL_{\univ} \otimes \omega_\pi^{m_0-1})) \in H^4(\bigQmap, \bQ)$$
where $\cL_{\univ}$ is the universal line bundle on the universal curve $\pi: \cC_{\univ} \xrightarrow{} \bigQmap$ and $\omega_\pi$ is the relative dualizing sheaf. Suppose we have a single spinning rational tail $T$ spining with rate $r$. Similar to the weight computations above, the relevant linearized line bundle to consider is $\cO_T(-1,-r)$, after isolating for the contribution from $T$. Taking the twice the second chern character gives $r^2$.

The computation is very similar for $s$ spinning rational tails $T_1, \ldots T_s$, each of degree $m_0$, and attached to $C_1$ at distinct points $p_1, \ldots p_s$. Each tail spins independently with rates $r_0, \ldots, r_s \in \bZ^s$, which gives $\sum_{i=1}^s r_i^2$.
\end{proof}

We are now able to define numerical invariants using the line bundles $\mathcal{M}_{a,b}$, $\mathcal{M}_{a,b}^\vee$ from in \Cref{D: line-bundle} and the norm from \Cref{D:norm =r^2}.

\begin{defn}\label{D:numerical-invariant}
For any non-degenerate filtration $f: \Theta \to \bigQmap$, we define the following numerical invariants
$$\nu_1(f) := \frac{2\wt_0(f^*\mathcal{M}_{a, b})}{\sqrt{f^*b}} \text{ \hspace{0.5cm} and \hspace{0.5cm} } \nu_2(f) := \frac{2\wt_0(f^*\mathcal{M}_{a, b}^\vee)}{\sqrt{f^*b}}$$
where $\mathcal{M}_{a, b} = \det(R\pi_*(\cL_{\univ}^a \otimes \omega_\pi^b))^\vee$ and $b$ is the norm on graded points given by $f^*b = \sum_{i} r_i^2$, where $\vec{r} = (r_1, \ldots, r_s)$ records the rates of the $\bG_m^s$ action on the spinning tails in the central fibre of the filtration.
\end{defn}

Using the explicit formulas for $\wt_0(f^*\mathcal{M}_{a,b})$ and $\wt_0(f^*\mathcal{M}_{a,b}^\vee)$ calculated in \Cref{E:LB-weight}, we obtain a more explicit formula for the numerical invariant as
$$\nu_1(f) = \dfrac{-\bigl(\sum_{i=1}^s r_i\bigr)}{\sqrt{\sum_{i=1}^s r_i^2}}(d'+1)(-1 + b - m_0 a) = - \nu_2(f)$$
where $s$ is the number of spinning tails in the central fibre, and $d' =am_0-b-1$. Following \cite[Definition 4.1.5]{halpernleistner2022structureinstabilitymodulitheory}, we now define semistability with respect to the numerical invariant $\nu_i$ for $i = 1, 2$.

\begin{defn}[$\Theta$-Semistability]
A point $p \in |\bigQmap|$ is \textbf{$\nu_i$-unstable} if there exists a non-degenerate filtration $f$ such that $\nu_i(f) < 0$ and $f(1) \simeq p$. Otherwise, $p$ is \textbf{$\nu_i$-semistable}. 
\end{defn}

\begin{thm}\label{T:theta-stratifications}
    The numerical invariants given in \Cref{D:numerical-invariant} define $\Theta$-stratifications of $\bigQmap$.
\end{thm}

\begin{proof}
    The numerators of our numerical invariants are defined with respect to classes in $H^2(\bigQmap,\bQ)$. Similarly, a class in $H^4(\bigQmap,\bQ)$ yields the norm on graded points. Since our stack $\bigQmap$ is quasi-compact, $\Theta$-reductive and $S$-complete, by  \cite[Theorem 5.6.21(1)]{halpernleistner2022structureinstabilitymodulitheory}, the numerical invariant associated to these cohomology classes defines a $\Theta$-stratification of $\bigQmap$.
\end{proof}

Note that for our numerical invariant defined above, when $d' = -1$ (at the wall $\epsilon_0$) as in \Cref{R: weight-vanishes}, the numerical invariant vanishes and the semistable locus is $\bigQmap$. Therefore, for the rest of this section, we consider $b \notin \{am_0, am_0+1\}$ in order to get non-trivial $\Theta$-stratifications.

\subsection{Destabilizing Filtrations}
We now demonstrate destabilizing filtrations for the line bundles $\mathcal{M}_{a,b}$ and $\mathcal{M}_{a,b}^\vee$ in order to characterize unstable points corresponding to each numerical invariant. These filtrations correspond to destabilizing objects on either side of the wall: $\epsilon_+$-stable quasimaps and $\epsilon_-$-stable quasimaps.

\paragraph{Destabilizing Filtrations for $\epsilon_+$-Stable Quasimaps.} \label{C:scaling filtration}
We first construct destabilizing filtrations with respect to the line bundle $\mathcal{M}_{a, b}$ for $\epsilon_+$-stable quasimaps that are not $\epsilon_-$-stable. Recall that such quasimaps allow degree $m_0$ rational tails, but basepoints must have length strictly less than $m_0$. Specifically, if a quasimap is $\epsilon_+$-stable but not $\epsilon_-$-stable, it must contain at least one degree $m_0$ rational tail.

Consider an $\epsilon_+$-stable quasimap with a degree $m_0$ rational tail $T$ attached to the main component of the curve. To construct a destabilizing filtration, we scale the sections defining the quasimap on $T$, pushing all basepoints toward $[1:0]$ on the tail, away from the attaching node $p = [0:1]$. More precisely, consider $t \in \mathbb{A}^1$ and scale the $x$ coordinate of the sections on $T$ by $t$. If the quasimap on $T$ is given by sections $(s_0(x, y), \ldots, s_N(x, y))$ of degree $m_0$, we scale them by $(s_0(tx, y), \ldots, s_N(tx, y))$. As $t \to 0$, this scaling forces all the sections to vanish at point $[1:0]$. In particular, $x$ does not divide at least one section, otherwise we would have a basepoint at the $[0:1]$ node. This results in a length $m_0$ basepoint on far end $[1:0]$ of the tail at the limiting point $t \to 0$.

Geometrically, this scaling produces a degeneration whose central fibre consists of the original main component together with a spinning rational tail $T$. The $\mathbb{G}_m$-action on the tail has negative spin rate, reflecting the direction of degeneration: as the parameter approaches $0$, the basepoints are pushed toward $[1:0]$.

\begin{figure}[H]
    \centering
\begin{tikzpicture}[scale=0.8]
  \tikzset{
    surface/.style={thick},
    cross_front/.style={gray!70, thick},
    cross_back/.style={gray!70, thick, dashed},
    dot_green/.style={circle, fill=cyan!40, inner sep=1.8pt},
    dot_blue/.style={circle, fill=blue!80!black, inner sep=1.8pt},
    prog_arrow/.style={-{Stealth[scale=1.5]}, line width=0.8pt}
  }

  \begin{scope}[shift={(0,0)}]
    \draw[surface] (0,0) ellipse (1.5 and 0.8);
    \draw[surface] (-1.8, 0.8) to[out=-53,in=53] (-1.8, -0.8);
    \draw[cross_front] (-0.5, 0.75) arc (90:270:0.25 and 0.75);
    \draw[cross_back] (-0.5, 0.75) arc (90:-90:0.25 and 0.75);
    
    \node[dot_green] at (0.1, 0.3) {};
    \node[dot_green] at (0.5, -0.2) {};
    \node[dot_green] at (1.0, 0.2) {};
    \node[dot_green] at (1.5, 0) {};
  \end{scope}

  \begin{scope}[shift={(5.5,0)}]
    \draw[surface] (0,0) ellipse (1.5 and 0.8);
    \draw[surface] (-1.8, 0.8) to[out=-53,in=53] (-1.8, -0.8);
    \draw[cross_front] (-0.5, 0.75) arc (90:270:0.25 and 0.75);
    \draw[cross_back] (-0.5, 0.75) arc (90:-90:0.25 and 0.75);
    
    \node[dot_green] at (0.8, 0.3) {};
    \node[dot_green] at (1.1, 0.1) {};
    \node[dot_green] at (1.0, -0.2) {};
    \node[dot_green] at (1.5, 0) {};
  \end{scope}

  \begin{scope}[shift={(11,0)}]
    \draw[surface] (0,0) ellipse (1.5 and 0.8);
    \draw[surface] (-1.8, 0.8) to[out=-53,in=53] (-1.8, -0.8);
    \draw[cross_front] (-0.5, 0.75) arc (90:270:0.25 and 0.75);
    \draw[cross_back] (-0.5, 0.75) arc (90:-90:0.25 and 0.75);
    
    \node[dot_blue] at (1.5, 0) {};
  \end{scope}

  \begin{scope}[shift={(0,-2.5)}]
    \draw[thick] (-2, 0) -- (13, 0);
    \draw[thick] (6.5, 0) node { $>$};
    \fill (11.0, 0) circle (2.5pt) node[below=3pt] {0};
    \node[anchor=west] at (13.2, 0.1) { $\mathbb{A}^1 / \mathbb{G}_m$};
  \end{scope}

\end{tikzpicture}
\caption{A destabilizing filtration for an $\epsilon_+$-stable quasimap. The light blue dots are basepoints of length $<m_0$, and the dark blue dot is a basepoint of length $m_0$ on a rational tail of degree $m_0$.}
\end{figure}
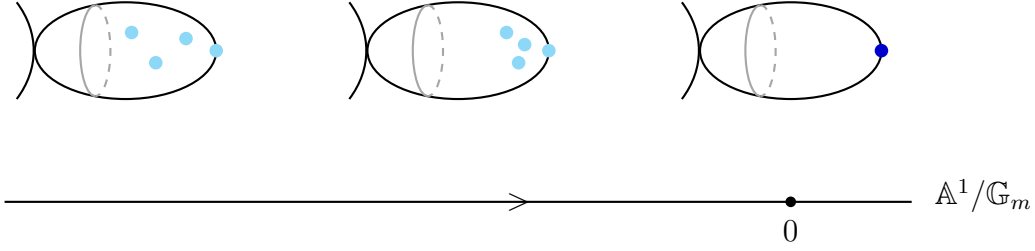

If the $\epsilon_+$-stable quasimap has multiple degree $m_0$ tails $T_1,\ldots,T_s$, we may perform this scaling independently on each tail. The resulting central fibre consists of the main component together with $s$ spinning tails $T_1,\ldots,T_s$, each carrying a length $m_0$ basepoint at the far end and spinning with rate $r_i < 0$. The independence of the tails means that the $\mathbb{G}_m^s$-action on the central fiber has weights $(r_1,\ldots,r_s)$ with each $r_i < 0$.

\begin{lem} \label{L:destabilizing-filt-minus} \label{L:epsilon-plus-dest-filt} The filtration constructed above satisfies $\nu_1(f) < 0$, hence is destabilizing for $\mathcal{M}_{a,b}$.

\begin{proof}
Since each $r_i < 0$, the sum $\sum r_i$ is negative, making the factor $-\sum r_i$ positive. Recall from \Cref{D:numerical-invariant}, we have 
$$\nu_1(f) = \dfrac{-\bigl(\sum_{i=1}^s r_i\bigr)}{\sqrt{\sum_{i=1}^s r_i^2}}(d'+1)(-1 + b - m_0 a)$$
so we examine the sign of the remaining factors. If $d' \ge 0$, we have $(-1 + b - m_0 a) < 0$, so the product $(d'+1)(-1 + b - m_0 a)$ is negative. If $d' \le -2$, we have $(-1 + b - m_0 a) > 0$, so the product $(d'+1)(-1 + b - m_0 a)$ is also negative. With $-\sum r_i > 0$, the product is negative, so $\nu_1(f) < 0$.
\end{proof}
\end{lem}

\paragraph{Destabilizing Filtrations for $\epsilon_-$-Stable Quasimaps.} \label{C:normal-cone-filtration} We now construct destabilizing filtrations for the dual line bundle $\mathcal{M}_{a,b}^\vee$. Recall that $\epsilon_-$-stable quasimaps allow basepoints of length $m_0$ but prohibit degree $m_0$ rational tails. Specifically, if a quasimap is $\epsilon_-$-stable but not $\epsilon_+$-stable, it must contain at least one length $m_0$ basepoint.

Consider an $\epsilon_-$-stable quasimap with a length $m_0$ basepoint $x \in C$. To obtain a destabilizing filtration for $\mathcal{M}_{a,b}^\vee$, we perform a deformation to the normal cone centered at $x$. 

To do this, consider the product $C \times \mathbb{A}^1$ where $\bA^1$ has coordinate $t$, and blow up the closed subscheme $\{x\} \times {0}$. The blow-up contains an exceptional divisor $T \cong \bP^1$.  Removing the proper transform of $C \times {0}$ yields a flat family where, for $t \neq 0$, the fiber is isomorphic to the original curve $C$. At $t = 0$, the fiber consists of the original curve (with the basepoint removed) together with a rational tail $T \cong \mathbb{P}^1$ attached at $x$. The $\mathbb{G}_m$-action on the family induces a $\mathbb{G}_m$-action on the central fiber with positive weight on the new rational tail, corresponding to a positive spin rate. 

\begin{figure}[H]
    \centering
\begin{tikzpicture}[scale=0.8]
  \tikzset{
    surface/.style={thick},
    cross_front/.style={gray!70, thick},
    cross_back/.style={gray!70, thick, dashed},
    dot_green/.style={circle, fill=cyan!40, inner sep=1.8pt},
    dot_blue/.style={circle, fill=blue!80!black, inner sep=1.8pt},
    prog_arrow/.style={-{Stealth[scale=1.5]}, line width=0.8pt}
  }

  \begin{scope}[shift={(0.5,0)}]
    \draw[surface] (0, 0.8) to[out=-53,in=53] (0, -0.8);
    \node[dot_blue] at (0.28, 0) {};
  \end{scope}

  \begin{scope}[shift={(2.0,0)}]
    \draw[surface] (0, 0.8) to[out=-53,in=53] (0, -0.8);
    \node[dot_blue] at (0.28, 0) {};
  \end{scope}

  \begin{scope}[shift={(3.5,0)}]
    \draw[surface] (0, 0.8) to[out=-53,in=53] (0, -0.8);
    \node[dot_blue] at (0.28, 0) {};
  \end{scope}

  \begin{scope}[shift={(8.5,0)}]
    \draw[surface] (0,0) ellipse (1.5 and 0.8);
    \draw[surface] (-1.8, 0.8) to[out=-53,in=53] (-1.8, -0.8);
    \draw[cross_front] (-0.5, 0.75) arc (90:270:0.25 and 0.75);
    \draw[cross_back] (-0.5, 0.75) arc (90:-90:0.25 and 0.75);
    
    \node[dot_blue] at (1.5, 0) {};
  \end{scope}

  \begin{scope}[shift={(0,-2.5)}]
    \draw[thick] (0, 0) -- (10.5, 0);
    \draw[thick] (4.5, 0) node { $>$};
    \fill (8.5, 0) circle (2.5pt) node[below=3pt] {0};
    \node[anchor=west] at (10.9, 0.1) { $\mathbb{A}^1 / \mathbb{G}_m$};
  \end{scope}

\end{tikzpicture}

\caption{A destabilizing filtration for an $\epsilon_-$-stable quasimap. The dark blue dots are basepoints of length $m_0$. A rational tail of degree $m_0$ appears in the fibre over 0.}

\end{figure}
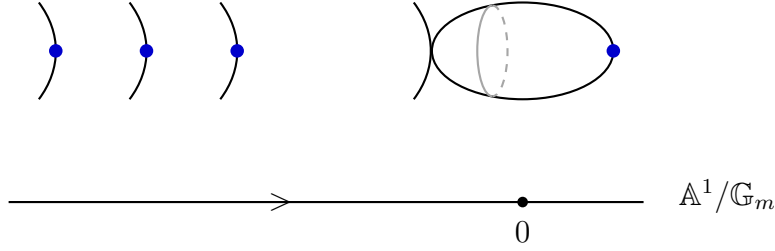

If the $\epsilon_-$-stable quasimap has multiple length $m_0$ basepoints $x_1, \ldots, x_s$, we may perform this procedure independently at each point. The resulting central fibre consists of the main component together with $s$ spinning tails $T_1,\ldots,T_s$, each carrying a length $m_0$ basepoint at its far end and spinning with rate $r_i > 0$. The independence of the tails means that the $\mathbb{G}_m^s$-action on the central fiber has weights $(r_1,\ldots,r_s)$ with each $r_i > 0$.

\begin{lem} \label{L:epsilon-minus-dest-filt}
The filtration constructed above satisfies $\nu_2(f) < 0$, hence is destabilizing for $\mathcal{M}_{a,b}^\vee$.
\end{lem}
\begin{proof}
For the dual line bundle, the numerical invariant is the negative of that for $\mathcal{M}_{a,b}$. From \Cref{D:numerical-invariant}, we have
$$\nu_2(f) = \dfrac{\bigl(\sum_{i=1}^s r_i\bigr)}{\sqrt{\sum_{i=1}^s r_i^2}}(d'+1)(-1 + b - m_0 a).$$
Since each $r_i > 0$, we have $\sum r_i > 0$. As established in \Cref{L:destabilizing-filt-minus}, $(d'+1)(-1 + b - m_0 a) < 0$ for all $d' \neq -1$. Thus $\nu_2(f) < 0$, confirming the filtration is destabilizing.
\end{proof}

\subsection{Semistable Loci} 
We show the semistable loci for the line bundles $\mathcal{M}_{a,b}$ and $\mathcal{M}_{a,b}^\vee$ are $\Qmapminus$ and $\Qmapplus$, respectively.

\begin{prop} \label{P:ss-containment}
The semistable locus for $\mathcal{M}_{a,b}$ is contained in $\Qmapminus$, and the semistable locus for $\mathcal{M}_{a,b}^\vee$ is contained in $\Qmapplus$.
\end{prop}
\begin{proof}
From \Cref{L:epsilon-plus-dest-filt}, we can destabilize all quasimaps that are $\epsilon_+$-stable and not $\epsilon_-$-stable for the line bundle $\mathcal{M}_{a,b}$. Similarily, from \Cref{L:epsilon-minus-dest-filt}, we can destabilize all quasimaps that are $\epsilon_-$-stable and not $\epsilon_+$-stable for the line bundle $\mathcal{M}_{a,b}^\vee$. We are left to analyze filtrations of $\epsilon_0$-semistable quasimaps that are not $\epsilon_-$ and $\epsilon_+$-stable. Such a quasimap must contain both a degree $m_0$ rational tail and a length $m_0$ basepoint. If all length $m_0$ basepoints are contained in a degree $m_0$ rational tail, the quasimap cannot degenerate further and we just have the trivial filtration, which can be destabilized on both sides of the wall. If there is length $m_0$ basepoint not contained in a degree $m_0$ rational tail, we can destabilize with respect to $\mathcal{M}_{a,b}^\vee$ via deformation to the normal cone as in \Cref{C:normal-cone-filtration}. Analogously, if there is a degree $m_0$ rational tail with no length $m_0$ basepoint on it, we can destabilize it with respect to $\mathcal{M}_{a,b}$ via scaling the quasimap sections as constructed in \Cref{C:scaling filtration}. 
\end{proof}

We have containment by exhibiting destabilizing filtrations, but we ultimately want to show that in fact the semistable loci for $\mathcal{M}_{a,b}$ and $\mathcal{M}_{a,b}^\vee$ are equal to $\Qmapminus$, and $\Qmapplus$. To do this, we analyze all possible filtrations of points of $\bigQmap$ to determine semistability. To analyze what possible filtrations of a quasimap can occur, we use the contraction morphism of \Cref{T: e_0-contraction}. For a  filtration $f: \Theta_K \to \bigQmap$, consider the diagram
\begin{center}
\begin{tikzcd}
\Theta_K \arrow[rr, "f"] \arrow[rrd] &  & \bigQmap \arrow[d, "\tilde{c}_{\bP^N}"] \\
                                     &  & \Qmapminus                             
\end{tikzcd}
\end{center}
where $\tilde{c}_{\bP^N}$ is the contraction morphism, taking a $\epsilon_0$-semistable quasimap to it's underlying $\epsilon_-$-stable model. All filtrations of points in $\bigQmap$ must factor through a fiber of the contraction morphism $\tilde{c}_{\bP^N}$ over some point in $\Qmapminus$, since any filtration of a point of $\Qmapminus$ is trivial. Thus the central fibre of any possible filtration of a point $f(1) \simeq p$ of $\bigQmap$ must be contained in the same fibre of the contraction morphism as $p$.
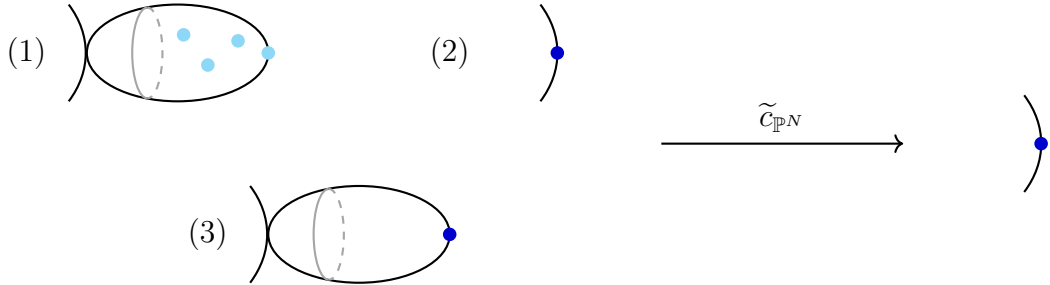
\begin{figure}[H]
\centering
\begin{tikzpicture}[scale = 0.8]

  \tikzset{
    surface/.style={thick},
    cross_front/.style={gray!70, thick},
    cross_back/.style={gray!70, thick, dashed},
    dot_green/.style={circle, fill=cyan!40, inner sep=1.8pt},
    dot_blue/.style={circle, fill=blue!80!black, inner sep=1.8pt},
    prog_arrow/.style={-{Stealth[scale=1.5]}, line width=0.8pt}
  }


  \begin{scope}[shift={(0,2)}]
    \draw[surface] (0,0) ellipse (1.5 and 0.8);
    \draw[surface] (-1.8, 0.8) to[out=-53,in=53] (-1.8, -0.8);
    \draw[cross_front] (-0.5, 0.75) arc (90:270:0.25 and 0.75);
    \draw[cross_back] (-0.5, 0.75) arc (90:-90:0.25 and 0.75);
    
    \node[dot_green] at (0.1, 0.3) {};
    \node[dot_green] at (0.5, -0.2) {};
    \node[dot_green] at (1.0, 0.2) {};
    \node[dot_green] at (1.5, 0) {};
    \node at (-2.5, 0) {(1)};
  \end{scope}

  \begin{scope}[shift={(6,2)}]
    \draw[surface] (0, 0.8) to[out=-53,in=53] (0, -0.8);
    \node[dot_blue] at (0.28, 0) {};
    \node at (-1.5, 0) {(2)};
  \end{scope}


  \begin{scope}[shift={(3, -1)}]
    \draw[surface] (0,0) ellipse (1.5 and 0.8);
    \draw[surface] (-1.8, 0.8) to[out=-53,in=53] (-1.8, -0.8);
    \draw[cross_front] (-0.5, 0.75) arc (90:270:0.25 and 0.75);
    \draw[cross_back] (-0.5, 0.75) arc (90:-90:0.25 and 0.75);
    
    \node[dot_blue] at (1.5, 0) {};
    \node at (-2.5, 0) {(3)};
  \end{scope}

\draw[thick, ->] (8, 0.5) -- (12, 0.5) node[midway, above] {$\widetilde{c}_{\mathbb{P}^N}$};

 \begin{scope}[shift={(14,0.5)}]
    \draw[surface] (0, 0.8) to[out=-53,in=53] (0, -0.8);
    \node[dot_blue] at (0.28, 0) {};
  \end{scope}

\end{tikzpicture}
\caption{A depiction of the fibres of the contraction morphism $\tilde{c}_{\bP^N}$ over a basepoint of length $m_0$. There are (1): $\epsilon_+$-stable quasimaps consisting of degree $m_0$ rational tails with possible basepoints of length $<m_0$ in light blue, (2): a unique $\epsilon_-$-stable quasimap with a basepoint of length $m_0$ in dark blue, and (3) a unique strictly $\epsilon_0$-semistable quasimap with spinning rational tail.}
\end{figure}

\begin{prop}\label{prop:ss=Q-minus}
The semistable locus for $\mathcal{M}_{a,b}$ is $\Qmapminus$.
\end{prop}
\begin{proof}

We have containment by \Cref{P:ss-containment}, so we are left to show that any filtration of an $\epsilon_-$-stable quasimap is semistable with respect to $\mathcal{M}_{a,b}$. The generic fibre of any non-degenerate filtration of an $\epsilon_-$-stable quasimap that is not $\epsilon_+$-stable must contain a length $m_0$ basepoint. Looking at the fibre of the contraction morphism yields that in the possible central fibres of the filtration, the length $m_0$ basepoint either perists (2) or becomes a spinning rational tail attached to the main component of the curve at the point where the basepoint was (3). For the non-degenerate filtration (3) where the basepoint becomes a spinning tail in the fibre over $0$, we must analyze which way the tail spins to determine semistability. The spinning tail in the central fibre must arise as a blowup of the surface $C_1 \times \Theta$, where $C_1$ is the fibre of the filtration over the point $1 \in \bA^1$. Analyzing the blowup of \Cref{C:normal-cone-filtration} locally in coordinates reveals that the rate of spin is strictly positive.
\end{proof}

\begin{prop}\label{prop:ss=Q-plus}
The semistable locus for $\mathcal{M}_{a,b}^\vee$ is $\Qmapplus$.
\end{prop}
\begin{proof}

We have containment by \Cref{P:ss-containment}, so we are left to show that any filtration of an $\epsilon_+$ stable quasimap is semistable with respect to $\mathcal{M}_{a,b}^\vee$. The generic fibre of any non-degenerate filtration of an $\epsilon_+$-stable quasimap that is not $\epsilon_-$-stable must contain a degree $m_0$ rational tail. Looking at the fibre of the contraction morphism yields that in the possible central fibres of the filtration, the degree $m_0$ rational tail either perists without spinning (1), or becomes a spinning rational tail (3). For the non-degenerate filtration (3) where the basepoint becomes a spinning tail in the fibre over $0$, we must analyze which way the tail spins to determine semistability. Since the length $m_0$ basepoint on the spinning tail must be disjoint from the node, the rational tail is forced to spin with a strictly negative rate, as discussed in \Cref{C:scaling filtration}.
\end{proof}

\bibliography{refs}{}
\bibliographystyle{alpha}

\noindent \textsc{Department of Mathematics, Cornell University, Ithaca, NY 14853} \\
\noindent \emph{Email address:} \texttt{srs382@cornell.edu}
\end{document}